\documentclass[11pt]{article}
\usepackage{graphicx}
\usepackage{subfigure}
\usepackage{epsfig}
\usepackage{amsmath}
\usepackage{amsthm}
\usepackage{amssymb}
\usepackage{lscape}
\usepackage{color}
\usepackage{tikz}
\usepackage{lineno}
\usepackage[ruled,boxed]{algorithm2e}
\usepackage{booktabs}
\usepackage{rotating}

\newtheorem{thm}{Theorem}[section]
\newtheorem{lmm}[thm]{Lemma}
\theoremstyle{definition}

\newtheorem{rem}{Remark}[section]

\begin{document}

\title{An Iterative Approach for Time Integration Based on Discontinuous Galerkin Methods}%
\author{Xiaozhou Li\footnotemark[1]~\footnotemark[2]
  \and Pietro Benedusi\footnotemark[1]
\and Rolf Krause\footnotemark[1]}

\renewcommand{\thefootnote}{\fnsymbol{footnote}}
\footnotetext[1]{Institute of Computational Science, Universit{\`a} della Svizzera italiana, Lugano, Switzerland. (\{xiaozhou.li, pietro.benedusi, rolf.krause\}@usi.ch)}
\footnotetext[2]{Corresponding author.}
\date{\today}

\maketitle

\begin{abstract}
 We present a new class of iterative schemes for solving initial value problems (IVP) based on discontinuous Galerkin (DG) methods.  Starting from the weak DG formulation of an IVP, we derive a new iterative method based on a preconditioned Picard iteration.  Using this approach, we can systematically construct explicit, implicit and semi-implicit schemes with arbitrary order of accuracy.  We also show that the same schemes can be constructed by solving a series of correction equations based on the DG weak formulation.  The accuracy of the schemes is proven to be $\min\{2p+1, K+1\}$ with $p$ the degree of the DG polynomial basis and $K$ the number of iterations.  The stability is explored numerically; we show that the implicit schemes are $A$-stable at least for $0 \leq p \leq 9$.  Furthermore, we combine the methods with a multilevel strategy to accelerate their convergence speed.  The new multilevel scheme is intended to provide a flexible framework for high order space-time discretizations and to be coupled with space-time multigrid techniques for solving partial differential equations (PDEs).  We present numerical examples for ODEs and PDEs to analyze the performance of the new methods.   
Moreover, the newly proposed class of methods, due to its structure, is also a competitive and promising candidate for parallel in time algorithms such as Parareal, PFASST, multigrid in time, etc.

\end{abstract}

\section{Introduction} 
The construction of efficient, stable and high-order numerical methods for the solution of initial value problems governed by ordinary differential equations has been studied extensively in past decades.  Existing methods for such problems can be classified, roughly speaking, into two groups.  The first group consists of discretization schemes based on the strong differentiation/integration formula of the initial value problem, which includes Runge-Kutta methods~\cite{Butcher:1987}, linear multi-step methods, and spectral deferred correction (SDC)~\cite{Dutt:2000}. This first group, in many respects, has become a mature subject and a dominant approach for solving both non-stiff and stiff problems~\cite{Hairer:1993, Hairer:1996}.  Furthermore, many techniques have been developed such as semi-implicit schemes~\cite{Ascher:1995, Kennedy:2003, Minion:2003} and the parallel in time algorithms~\cite{Ketcheson:2014,Lions:2001, Emmett:2012}. In particular the development of parallel in time algorithms is now a growing field. In fact time-parallelism allows to extend the scalability of a software using domain decomposition in the time direction aside from the space one \cite{Gander2015_Review,MinionEtAl2015}.

The second group consists of methods based on the weak Galerkin formulation of the initial value problem.  The earliest developments of Galerkin approaches have been introduced by Argyris and Scharpf~\cite{Argyris:1969}, Fried~\cite{Fried:1969}, and Hulme~\cite{Hulme:1972:1, Hulme:1972:2} with continuous finite element methods more than 40 years ago.  After that, the continuous finite element methods as time discretizations have been intensively studied by many authors, for example by Betsch and Steinmann in their series of work~\cite{Betsch:2000, Betsch:2001, Betsch:2002, Gross:2005}.

The focus of this paper is the DG method, which falls from the second group.  The first analysis of the DG method applied to ODEs was done by Lesaint and Raviart~\cite{Lesaint:1974} in 1974, right after the introduction of the DG methods in 1973 by Reed and Hill~\cite{Reed:1973}.  In~\cite{Lesaint:1974}, Lesaint and Raviart showed that the DG method with polynomials of degree $p$ is A-stable of order $2p+1$ at the mesh points, and proved the results for linear cases.  This property of ``order $2p+1$ at the mesh points'' was later called the superconvergence property.  A rigorous proof for non-linear cases was given by Adjerid, Devine, Flaherty and Krivodonova~\cite{Adjerid:2002} in 2002.  In the extension of the standard DG approach, an $\alpha$-averaging DG method for ODEs was introduced by Delfour, Hager and Trochu~\cite{Delfour:1981}; they approximated the solution $u$ at time $t_n$ by taking the average of the jump: $u(t_n) \approx \alpha_n u_h(t_n^{-}) + (1 - \alpha_n) u_h(t_n^{+})$.  For piecewise constant approximations, the values $\alpha_n = 0, \frac{1}{2}, 1$ correspond, respectively, to Euler's explicit, improved, and implicit scheme.  

Since then, many authors have studied the derivation of time discretization schemes based on DG approaches.  Examples are one-step methods, such as implicit Runge-Kutta method, and multistep methods, such as Adams-Bashforth and Adams-Moulton schemes, see~\cite{Delfour:1986, Bottasso:1997, Estep:2002, Zhao:2014}.  Although DG methods have attractive features, such as excellent stability property (A-stability) and high-order accuracy  (superconvergence) for solving initial value problems, there are still notable challenges:
\begin{itemize}
    \item for nonlinear ODEs, Galerkin approaches lead to nonlinear systems of equations, which usually are not trivial to solve.  Especially, for the time discretization of nonlinear PDEs, methods for their solutions tend to be rather difficult to code and computationally expensive.
    \item The fully implicit nature of Galerkin schemes makes them less efficient and attractive for non-stiff problems compared to the classical time integration schemes.
    \item Explicit-implicit discretizations, which are capable of treating the non-stiff terms explicitly and the stiff term implicitly, have not been developed yet.
    \item In the context of time discretizations, the now popular
     parallel in-time-algorithms are not yet used in combination with the Galerkin approach.
\end{itemize}

The above points are the main reason why the DG time stepping method is not as widely used, in the sense of applications and publications, as classic time stepping methods such as Runge-Kutta methods.  Besides its usage as a time stepping method, it is worth to mention that the DG method has been used as time discretization of space-time finite element methods, cf.~\cite{Jamet:1978, Eriksson:1985, Eriksson:1987, Eriksson:1991, Eriksson:1995:1, Eriksson:1995:2, Eriksson:1995:3}.  However, also in the context of space-time discretization, one has to face the above challenges.  For example, in~\cite{Klaij:2006} the authors coupled the space-time discretization with additional pseudo-time stepping methods to generate explicit schemes and deal the nonlinearity.  

In order to address these challenges, here we derive a new class of time stepping schemes based on the standard DG time stepping method introduced by Lesaint and Raviart~\cite{Lesaint:1974} in 1974.  We start from the standard weak DG formulation, where the DG approximation $u_h$ is constructed in the nodal form~\cite{Hesthaven:2008} based on the right Gau\ss-Radau points.  The solution $u$ at $t_n$ is approximated by taking the upwind flux: $u(t_n) = u_h(t_n^{-})$.  
Furthermore, to avoid the trouble of solving the fully implicit nonlinear system generated by the weak Galerkin formulation, we introduce different iterative methods, which lead to explicit, implicit and semi-implicit schemes.  

The theoretical analysis is also given; the new iterative schemes have accuracy order of $\min\{2p+1, K+1\}$ with $p$ the degree of polynomial basis and $K$ the number of iterations.  The new schemes preserve the superconvergence property of the DG method~\cite{Adjerid:2002}.  Also, their stability is explored numerically; the implicit schemes are demonstrated to be $A$-stable at least for $0 \leq p \leq 9$.  Our proposed schemes are intended to:
\begin{itemize}
    \item be combined with the method of lines approach to yield a flexible framework for high order space-time methods for partial differential equations;
    \item be combined with the spatial multigrid methods to generate a framework for space-time multigrid methods;
    \item be combined with the parallel in time algorithms, such as Parareal, PFASST, etc. 
\end{itemize}
Therefore, a general multilevel strategy based on the full approximation scheme (FAS) and basic adaptive strategies are presented in this paper as starting points for future developments.  

This paper is organized as follows. In Section 2, we review the DG method for ODEs and its relevant properties.  In Section 3, we derive a new class of iterative time stepping schemes based on the DG methods and the error estimates are presented.  Furthermore, a general multilevel strategy based on full approximation schemes is introduced in Section 4 together with a brief description of adaptive strategies.  Numerical examples are given in Section 5. The conclusions are presented in Section 6.

\section{Background}  
In this section, we briefly review the weak formulation and necessary properties of the standard DG methods for the initial value problem: 
\begin{equation}  
    \label{eq:ivp}
       \left\{ \begin{array}{l} 
         u_t  = f(t, u(t)), \,\, t \in [0, T] \\
         u(0) = u_0
       \end{array}\right.. 
\end{equation}
For convenience, here we consider $u_0 \in \mathbb{R}$, $u: \mathbb{R} \mapsto \mathbb{R} $ and $f: \mathbb{R}\times\mathbb{R} \mapsto \mathbb{R}$.  Systems of ODEs can be addressed in a similar way.

To derive the DG weak formulation, we divide the time interval $[0, T]$ into $N$ subintervals by means of the partition $0 = t_0 < t_1 < \cdots < t_n < \cdots < t_N = T$.  Let $I_n = [t_n, t_{n+1}]$, $\Delta t_n = t_{n+1} - t_n$, $h = \max\limits_n\Delta t_n$. We denote by $u_n \approx u(t_n)$ the respective approximation for $u$ at $t_n$ resulting from the DG schemes described below. 

The DG approximation space is given by 
\begin{equation}
  \label{eq:space}
    V^p_h = \left\{v_h: v_h\big\vert_{I_n} \in \mathbb{P}^p(I_n)\right\},
\end{equation}
where $\mathbb{P}^p$ denote the space of all polynomials of degree $\leq p$.
On each interval $I_n$, we construct the DG approximation $u_h$ in the nodal form,
\begin{equation} 
  \label{eq:uh}
  u_h(t) = \sum\limits_{m = 0}^p u_{n,m}\ell_{n,m}(t), \qquad t \in [t_n, t_{n+1}],
\end{equation}
where $\left\{\ell_{n,m}(t)\right\}$ is the basis of Lagrange polynomials of degree $p$ with the right Gau\ss-Radau points ${\left\{t_{n,m}\right\}}_{m=0\ldots p}$ (see the following notes) over the interval $I_n$.  Due to the discontinuous nature of this approach, at the mesh point $t_{n}$, the DG approximation $u_h$ has two values: the limiting values from the left $u_h(t_{n}^{-})$ and from the right $u_h(t_{n}^{+})$ (see Figure~\ref{fig:grid}), which in general will be different. 

\begin{figure}[!ht] 
    \centering
    \includegraphics[width=.8\textwidth]{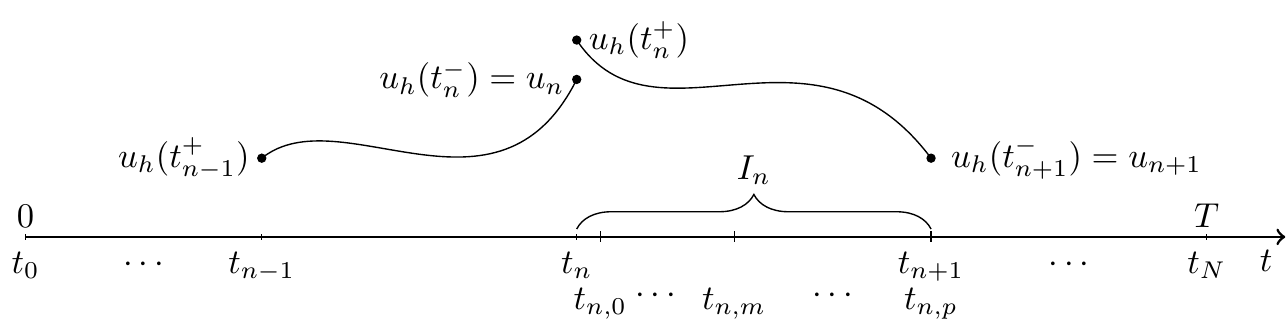}
    \begin{center}
    \caption{\label{fig:grid}
    The discretization of the time axis and the corresponding notation is shown.  Also the Gau\ss-Radau points $t_{n,m}$ are depicted for the interval $I_n$.   
}
     \end{center}
\end{figure}

In this paper, as in the original work of Lesaint and Raviart~\cite{Lesaint:1974}, the numerical approximation for $u(t_{n+1})$ is given by
\begin{equation}
  \label{eq:un}
  u_{n+1} = u_h(t_{n+1}^{-}) = \sum\limits_{m = 0}^p u_{n,m}\ell_{n,m}(t_{n+1}) = u_{n,p}.
\end{equation}
Here, we note:
\begin{itemize}
  \item The choice of the nodal form~\eqref{eq:uh} is convenient for quadrature purpose and crucial for the further derivations in this paper.  For more details about the nodal DG method and its comparison with modal DG method, we refer to Hesthaven and Warburton~\cite{Hesthaven:2008}.
  \item The Gau\ss-Radau points are the quadrature points of the Gau\ss-Radau quadrature, which requires $p+1$ points and reproduces exactly all polynomials of degree $2p$, see~\cite{Handbook:1992}.  Here, ``right'' Gau\ss-Radau points means that the right endpoint is included in the $p+1$ points ${\left\{t_{n,m}\right\}}_{m=0\ldots p}$,  
    \[
      t_n < t_{n,0} < t_{n,1} < \cdots < t_{n,p} = t_{n+1}.
    \]
    An explanation for this choice will be provided in Section 3.
 
\end{itemize}

Once the approximation space has been chosen, as usual, we multiply with test functions, $v_h \in V_h^p$,  on both sides of~\eqref{eq:ivp}, and integrate by parts:
\begin{equation}  
    \label{eq:weak}
    -\int_{t_n}^{t_{n+1}} u_h {(v_h)}_t \,dt + u_h(t_{n+1}^{-})v_h(t_{n+1}^{-}) - u_h(t_{n}^{- })v_h(t_{n}^{+}) = \int_{t_n}^{t_{n+1}} f(t,u_h(t))v_h \,dt, 
\end{equation}
where $u_h(t_{n+1}^{-})$ is the upwind flux and $u_h(t_n^{-})$ ($=u_n$) is passed from the previous interval $I_{n-1}$ as the initial value.

By inserting the nodal representation of $u_h$ we obtain, for $0 \leq j \leq p$, 
\[
    -\sum\limits_{m=0}^p u_{n,m}\int_{t_n}^{t_{n+1}} \ell_m(t)\ell'_j(t)\, dt + u_{n,p}\delta_{jp} - u_h(t_{n}^{-})\delta_{j0} = \int_{t_n}^{t_{n+1}}f(t,u_h(t))\ell_j(t)\,dt.
\]
For simplicity we consider in first place the linear case, $f(t,u_h(t)) = \lambda u_h(t)$, then, by transformation to the reference interval $[-1,1]$ we get
\begin{equation}    
    \label{eq:1}
    L U + \lambda\frac{\Delta t_n}{2} M U + B = 0, 
\end{equation}
with 
\begin{equation} 
  \label{eq:LM}
  L_{i,j} = \int_{-1}^{1} \ell'_i(t)\ell_j(t)\,dt - \delta_{ip}\delta_{jp}, \qquad M_{i,j} = \int_{-1}^{1} \ell_i(t)\ell_j(t) \,dt,
\end{equation}
where $\{\ell_i(t)\}$ are the Lagrange polynomials of degree $p$ on the reference interval $[-1,1]$, and
\[
  U = {[u_{n,0}, u_{n,1}, \ldots, u_{n,p}]}^T,\qquad B = {u_h(t_n^{-})\left[\ell_0(-1), \ell_1(-1), \ldots, \ell_p(-1)\right]}^T.
\]
By solving system~\eqref{eq:1}, we have the standard discontinuous Galerkin approximation of the test equation,
\[
  \left\{ \begin{array}{l}
    U = - {\left(L + \frac{\lambda\Delta t_n}{2} M\right)}^{-1}B \\
    u_{n+1} = u_{n,p}
  \end{array}\right..
\]
Extensive results on the properties of DG methods are available in the literature, see~\cite{Lesaint:1974, Delfour:1981, Delfour:1986, Johnson:1986, Richter:1988}.  Before presenting our new method, let us first collect some properties of DG methods, which will be useful later on.  
\begin{lmm}\label{lmm:stable} (Stability.)
  The DG method~\eqref{eq:weak} for IVP~\eqref{eq:ivp} is $A$-stable. 
\end{lmm}
\begin{proof}
  See Theorem 2 in Lesaint and Raviart~\cite{Lesaint:1974}.
\end{proof}

\begin{lmm}\label{lmm:super} (Superconvergence.)
    Denote by $u_h$ the DG approximation of degree $p$ for the IVP~\eqref{eq:ivp} ($f(t) \in \mathcal{C}^{2p+1}$) on the interval $I_n = [t_n, t_{n+1}]$, and ${\left\{t_{n,m}\right\}}_{m=0\ldots p}$ the $(p+1)$ Gau\ss-Radau points on $I_n$.  Then we have the local truncation error:
    \[
        u(t_{n,m}) - u_h(t_{n,m}) = \mathcal{O}\left(h^{p+2}\right), \quad 0 \leq m \leq p-1.
    \]
    and at the end point $u_h(t_{n+1}^{-}) = u_h(t_{n,p})$,
    \[
        u(t_{n+1}) - u_h(t_{n+1}^{-}) = \mathcal{O}\left(h^{2p+2}\right).
    \]
\end{lmm}
\begin{proof}
  See Theorem 5 in Adjerid et al.~\cite{Adjerid:2002}.
\end{proof}

As a time stepping method, the DG method has two attractive features: Lemma~\ref{lmm:stable} provides excellent stability ($A$-stability) and Lemma~\ref{lmm:super} high-order accuracy (the global error is of order $2p+1$).  However, the challenges arise when dealing with nonlinear ODEs.  If $f(u)$ is a nonlinear function, then the term $\lambda\frac{\Delta t_n}{2} M U$ becomes $\frac{\Delta t_n}{2}F(U)$, where
\begin{equation} 
    \label{eq:f}
    F(U) = {\left[\int_{-1}^1 f\left(t, \sum\limits_{m=0}^p u_{n,m}\ell_m(t)\right)\ell_0(t)\,dt, \ldots, \int_{-1}^1 f\left(t, \sum\limits_{m=0}^p u_{n,m}\ell_m(t)\right)\ell_p(t)\,dt\right]}^T.
\end{equation}
This leads to the nonlinear system 
\begin{equation} 
  \label{eq:non}
  L U + \frac{\Delta t_n}{2}F(U) + B = 0, 
\end{equation}
which in general can not be easily solved.  In the next section, we will construct an iterative approach to solve the system~\eqref{eq:non} and derive a new class of iterative schemes based on the DG weak form~\eqref{eq:weak}.   

\section{A New Class of Iterative Schemes for DG Methods}
In general, $f(t,u)$ in~\eqref{eq:f} can not be integrated analytically, for this reason, we use numerical quadrature to approximate the integral, see~\cite{Delfour:1981, Delfour:1986, Zhao:2014}.  The choice of quadrature is critical for the final scheme, for example, it can facilitate the derivations of the multi-steps rule, Runge-Kutta methods and hybrid methods from the DG methods, see~\cite{Delfour:1986, Zhao:2014}.  However, only few choices preserve the superconvergence property from Lemma~\ref{lmm:super}.  In this paper, in order to preserve the superconvergence property, we choose the Gau\ss-Radau quadrature to construct the DG approximation~\eqref{eq:uh}.  


As consequence of the nodal DG form~\eqref{eq:uh}, we can use the collected Gau\ss-Radau points for quadrature directly, 
\begin{equation} 
  \label{eq:fw}
  F(U) \approx F_\omega(U) = {\left[\omega_0 f_{n,0}, \omega_1 f_{n,1}, \ldots, \omega_p f_{n,p}\right]}^T = \operatorname{diag}\{\omega_0,\ldots,\omega_p\}{\left[f_{n,0}, \ldots, f_{n,p}\right]}^T,
\end{equation}
and the integrals in \eqref{eq:LM} can be evaluated exactly as 
\[
  L_{i,j} = \int_{-1}^{1} \ell'_i(t)\ell_j(t)\,dt - \delta_{ip}\delta_{jp} = \omega_j \ell'_i(t_j) - \delta_{ip}\delta_{jp}, \qquad M_{i,j} = \int_{-1}^{1} \ell_i(t)\ell_j(t)\,dt = \omega_j \delta_{ij},
\]
where $f_{n,m} = f(t_{n,m}, u_{n,m})$ and $t, \omega$ are the correspond right Gau\ss-Radau points and weights over the reference domain $[-1,1]$.  

By means of the numerical quadrature, we simplify the nonlinear system~\eqref{eq:non} to
\begin{equation} 
  \label{eq:non2}
  L U + \frac{\Delta t_n}{2}F_\omega(U) + B = 0.
\end{equation}
However, it is still a fully implicit system, and a nonlinear system has to be solved.  We emphasize that the nonlinearities are now localized, thus allowing for the construction of dedicated solution methods.

\subsection{A Simple Iterative Approach}
A naive way to solve system~\eqref{eq:non2} is to use a fixed point iteration. 

\noindent{\bf Explicit DG Scheme:}
\begin{equation}  
    \label{eq:DG-Ex}
    U^{k+1} = -\frac{\Delta t_n}{2}L^{-1}F_\omega(U^k) - L^{-1}B,
\end{equation}
where $U^0$ can be obtained, for example, by using the explicit Euler method.  

Clearly, the iterative scheme~\eqref{eq:DG-Ex} is an explicit scheme and its stability is guaranteed when $\Delta t_n \to 0$.  Although scheme~\eqref{eq:DG-Ex} avoids to solve a nonlinear system, we will show later that both its stability and convergence speed are not very satisfactory.  However, before going further to improve scheme~\eqref{eq:DG-Ex}, we first present its error estimates. 

\begin{lmm}\label{lmm:DG-Ex0}
  Denote $U^K = {[u^K_{n,0},\ldots,u^K_{n,p}]}^T$ the $K$-th iteration of scheme 
    \begin{equation} 
      \label{eq:DG-Ex0}
      U^{k+1} = -\frac{\Delta t_n}{2}L^{-1}F(U^k) - L^{-1}B,
    \end{equation}
    with $F(U)$ defined in~\eqref{eq:f} and $u_h$ the DG approximation of degree $p$ for the IVP~\eqref{eq:ivp} ($f(t) \in \mathcal{C}^{2p+1}$), then we have 
    \[
        u_h(t_{n,m}) - u_{n,m}^K = \mathcal{O}\left(h^{K+2}\right),\qquad m = 0,\dots,p,
    \]
    where $h = \max\limits_{n} \Delta t_n$, and $U^0$ is obtained by using the explicit Euler method.
\end{lmm}
\begin{proof}
  First, the DG solution $U_h = {[u_h(t_{n,0}), \ldots, u_h(t_{n,p})]}^T$ satisfies the weak formulation
    \[
        L U_h + \frac{\Delta t_n}{2}F(U_h) + B = 0,
    \]
    by comparing it to the iteration scheme 
    \[
        L U^{k+1} = -\frac{\Delta t_n}{2}F(U^k)-B,
    \]
    we have 
    \[
        L (U_h - U^{k+1}) = \frac{\Delta t_n}{2}\left(F(U^k) - F(U_h)\right).
    \]
    For using induction in $k$, we assume, 
    \[
        u_h(t_{n,m}) - u^{k}_{n,m}  = \mathcal{O}(h^{k+2}), \qquad 0 \leq m \leq p.
    \]
    For $k = 0$, the initial guess $U^0$ is given by the explicit Euler method, we have 
    \[
        u_h(t_{n,m}) - u^{0}_{n,m} = \mathcal{O}(h^{2}), \qquad 0 \leq m \leq p.
    \]
    For $k \geq 1$, we denote ${\left\{\cdot\right\}}_j$ the $j$-th component of a vector, then
    \[
      {\left\{\frac{\Delta t_n}{2}\left(F(U^k) - F(U_h)\right)\right\}}_j = \mathcal{O}(h|U^k - U_h|) = \mathcal{O}(h^{k+3}), \qquad 0 \leq j \leq p.
    \]
    Therefore,
    \[
      u_h(t_{n,m}) - u^{k+1}_{n,m} = {\left\{\frac{\Delta t_n}{2}L^{-1}\left(F(U^k) - F(U_h)\right)\right\}}_m = \mathcal{O}(h^{k+3}), \qquad k = 0,\dots,K-1,
    \]
    here $L^{-1}$ is a constant matrix which only depends on $p$.
\end{proof}

\begin{lmm}\label{lmm:DG-Ex1}
  Denote $U^K = {[u^K_{n,0},\ldots,u^K_{n,p}]}^T$ the $K$-th iteration of scheme~\eqref{eq:DG-Ex} and $u_h$ the DG approximation of degree $p$ for the IVP~\eqref{eq:ivp} ($f(t) \in \mathcal{C}^{2p+1}$), then we have 
    \[
        u_h(t_{n,m}) - u_{n,m}^K = \mathcal{O}\left(h^{\min\{2p+2, K+2\}}\right),
    \]
    where $h = \max\limits_{n} \Delta t_n$, and $U^0$ is obtained by using the explicit Euler method.
\end{lmm}
\begin{proof}
    First, the DG solution $U_h = {[u_h(t_{n,0}), \ldots, u_h(t_{n,p})]}^T$ satisfies 
    \[
        L U_h + \frac{\Delta t_n}{2}F(U_h) + B = 0,
    \]
    by comparing it to the iteration scheme 
    \[
        L U^{k+1} = -\frac{\Delta t_n}{2}F_\omega(U^k)-B,
    \]
    we have 
    \[
        L (U_h - U^{k+1}) = \frac{\Delta t_n}{2}\left(F_\omega(U^k) - F(U_h)\right).
    \]
    By using induction on $k$, we assume, 
    \[
        u_h(t_{n,m}) - u^{k}_{n,m}  = \mathcal{O}(h^{\min{\{2p+2,k+2\}}}), \qquad 0 \leq m \leq p.
    \]
    Since 
    \[
        F_\omega(U^k) - F(U_h) = F_\omega(U^k) - F_\omega(U_h) + F_\omega(U_h) - F(U_h),
    \]
    and by Gau\ss-Radau quadrature ($p+1$ points), see~\cite{Kambo:1970},
    \[
      {\left\{F_\omega(U_h) - F(U_h)\right\}}_j = \mathcal{O}(h^{2p+1}),
    \]
    with
    \[
      {\left\{F_\omega(U^k) - F_\omega(U_h)\right\}}_j = \mathcal{O}(|U^k - U_h|) = \mathcal{O}(h^{\min{\{2p+2,k+2\}}}), \qquad 0 \leq j \leq p,
    \]
    we have 
    \[
      {\left\{\frac{\Delta t_n}{2}\left(F_\omega(U^k) - F(U_h)\right)\right\}}_j = \mathcal{O}(h^{\min{\{2p+2,k+3\}}}) + \mathcal{O}(h^{2p+2})= \mathcal{O}\left(h^{\min\{2p+2, k+3\}}\right), \quad 0 \leq j \leq p.
    \]
    Therefore,
    \[
      u_h(t_{n,m}) - u^{k+1}_{n,m} = {\left\{\frac{\Delta t_n}{2}L^{-1}\left(F_\omega(U^k) - F(U_h)\right)\right\}}_m = \mathcal{O}\left(h^{\min\{2p+2, k+3\}}\right), \quad 0\leq m\leq p, \quad 0 \leq k \leq K-1,
    \]
    as mentioned in Lemma~\ref{lmm:DG-Ex0}, $L^{-1}$ is a constant matrix.
\end{proof}

\begin{lmm}\label{lmm:DG-Ex2} (Local Truncation Error.)
  Denote $U^K = {[u^K_{n,0},\ldots,u^K_{n,p}]}^T$ the $K$-th iteration of scheme~\eqref{eq:DG-Ex} and by $u$ the exact solution for the IVP~\eqref{eq:ivp} ($f(t) \in \mathcal{C}^{2p+1}$), then we have 
  \[
    u(t_{n,m}) - u^K_{n,m} = \mathcal{O}\left(h^{\min\{p+2, K+2\}}\right), \quad 0 \leq m \leq p-1.
  \]
  and at the end point $t_{n+1} = t_{n,p}$,
  \[
    u(t_{n+1}) - u^K_{n,p} = \mathcal{O}\left(h^{\min\{2p+2, K+2\}}\right),
  \]
  where $h = \max\limits_{n} \Delta t_n$.
\end{lmm}
\begin{proof}
  By Lemma~\ref{lmm:super}, we have 
  \[
    u(t_{n,m}) - u_h(t_{n,m}) = \mathcal{O}\left(h^{p+2}\right), \quad 0 \leq m \leq p-1.
  \]
  and at the end point
  \[
    u(t_{n+1}) - u_h(t_{n,p}) = \mathcal{O}\left(h^{2p+2}\right).
  \]
  From Lemma~\ref{lmm:DG-Ex1}, 
  \[
    u_h(t_{n,m}) - u^{K}_{n,m} = \mathcal{O}\left(h^{\min\{2p+2, K+2\}}\right).
  \]
  By applying the triangular inequality, we finish the proof.  
\end{proof}

\begin{rem}
  By comparing Lemma~\ref{lmm:DG-Ex0} and Lemma~\ref{lmm:DG-Ex1}, we see that the order of accuracy of the iterative approach to the DG approximation is bounded by the accuracy order of the used numerical quadrature and the number of iterations.  Therefore, due to the superconvergence property given by Lemma~\ref{lmm:super}, the right Gau\ss-Radau points will be the proper choice.  In fact, if the more popular Gau\ss-Lobatto points are used, then one order of accuracy will be lost.
\end{rem}

\subsection{A New Iterative Approach}
As mentioned earlier, the explicit scheme~\eqref{eq:DG-Ex} has unsatisfactory stability properties and low convergence speed, and we also want to derive implicit and semi-implicit schemes.  To derive those desired iterative schemes, we consider the general preconditioned iteration scheme as 
\begin{equation}
    \label{eq:preDG}
    U^{k+1} = U^k - P^{-1}\left(U^{k} + \frac{\Delta t_n}{2}L^{-1}F_\omega(U^k) + L^{-1}B\right).
\end{equation}
It is obvious that when $P = I$, we have the explicit scheme~\eqref{eq:DG-Ex}.  For simplicity, we denote $F_\omega U \equiv F_\omega(U)$.  Then, we obtain the original DG weak scheme by choosing $P = I + \frac{\Delta t_n}{2}L^{-1}F_\omega$.  In order to reduce the complexity, in this section, we consider $P = I + \frac{\Delta t_n}{2}L_\Delta^{-1}F_\omega$ where $L_\Delta$ is an approximation of $L$. 

As mentioned earlier in~\eqref{eq:1}, the matrix $L$ is given by
\[
  L_{i,j} = \int_{-1}^{1} \ell'_i(t)\ell_j(t)\,dt - \delta_{ip}\delta_{jp}, \qquad 0 \leq i,j \leq p,
\]
which mainly arises form ${\displaystyle \int_{t_n}^{t_{n+1}} \sum\limits_{m = 0}^p u_{n,m}\ell_m(t)\ell'_j(t)\,dt}$.  Here, by considering piecewise constant approximations
\[
  \sum\limits_{m =0}^p u_{n,m}\ell_m(t) \approx \sum\limits_{m=0}^{p-1} u_{n,m}\chi[t_{n,m},t_{n,m+1}],
\]
where $\chi$ is the standard characteristic function.  We have
\begin{align*}  
    &\, \int_{t_n}^{t_{n+1}} \sum\limits_{m = 0}^p u_{n,m}\ell_m(t)\ell'_j(t)\,dt \approx  \int_{t_n}^{t_{n+1}}\left\{\sum\limits_{m=0}^{p-1} u_{n,m}\chi[t_{n,m},t_{n,m+1}]\ell'_j(t)\right\}\,dt \\
    = &\, \sum\limits_{m=0}^{p-1}u_{n,m}\int_{t_{n,m}}^{t_{n,m+1}}\ell'_j(t)\,dt = \sum\limits_{m=0}^{p-1}u_{n,m}\left(\ell_j(t_{n,m+1}) - \ell_j(t_{n,m})\right).
\end{align*}
This leads to an approximation of $L$ given by 
\[
    L_\Delta = \left[\begin{array}{ccccc}
            -1  &    & & & \\
            1  & -1 & & &\\
               &\ddots & \ddots & &\\
               & & 1 & -1 & \\
               &&  & 1 & \color{red}-1
    \end{array}\right],
\]
where the red component is contributed from the $- \delta_{ip}\delta_{jp}$ term in~\eqref{eq:1}.  Substituting the $L_\Delta$ back into the preconditioned scheme~\eqref{eq:preDG}, we get
\[
    \left(I + \frac{\Delta t_n}{2}L_\Delta^{-1} F_\omega \right)U^{k+1} = \left(I + \frac{\Delta t_n}{2}L_\Delta^{-1} F_\omega \right)U^{k} - \left(U^{k} + \frac{\Delta t_n}{2} L^{-1}F_\omega(U^k) + L^{-1}B\right)
\]
$\Longrightarrow$
\[
    L_\Delta U^{k+1} + \frac{\Delta t_n}{2}F_\omega(U^{k+1}) = \frac{\Delta t_n}{2}F_\omega(U^k) - \frac{\Delta t_n}{2} L_\Delta L^{-1}F_\omega(U^k) - L_\Delta L^{-1}B.
\]
For convenience, we denote $B = {[b_0, \ldots, b_p]}^T$ and $L^{-1} = \left\{l^{-1}_{i,j}\right\}$.  Then we write the above formula componentwise for $U^k = {[u_{n,0}^k, \ldots, u_{n,p}^k]}^T$:
\begin{align*}  
    u^{k+1}_{n,0} & = \frac{\Delta t_n}{2}\omega_0\left(f(u^{k+1}_{n,0}) - f(u^{k}_{n,0})\right)  - \sum\limits_{j=0}^p l^{-1}_{0j}  \left(\frac{\Delta t_n}{2} \omega_j f(u^{k}_{n,j}) + b_j\right) \\
    u^{k+1}_{n,m+1} & = u^{k+1}_{n,m} + \frac{\Delta t_n}{2}\omega_{m+1}\left(f(u^{k+1}_{n,m+1}) - f(u^{k}_{n,m+1})\right) \\
                  &\qquad\qquad + \sum\limits_{j=0}^p \left(l^{-1}_{m,j} - l^{-1}_{m+1,j}\right)\left(\frac{\Delta t_n}{2} \omega_j f(u^{k}_{n,j}) + b_j\right), \qquad 0 \leq m \leq p-1. 
\end{align*}
Furthermore, if we define a matrix $\tilde{L} = \left\{\tilde{l}_{i,j}\right\}$ as 
\[
    \tilde{L} = \left[\begin{array}{rrrr}
            - l^{-1}_{0,0} & - l^{-1}_{0,1} & \cdots & -l^{-1}_{0,p} \\
            l^{-1}_{0,0} - l^{-}_{1,0} & l^{-1}_{0,1} - l^{-1}_{1,1} & \cdots & l^{-1}_{0,p} - l^{-1}_{1,p} \\
            \vdots & \vdots & \ddots & \vdots \\
            l^{-1}_{p-1,0} - l^{-}_{p,0} & l^{-1}_{p-1,1} - l^{-1}_{p,1} & \cdots & l^{-1}_{p-1,p} - l^{-1}_{p,p} 
    \end{array}\right],
\]
and a simple analysis shows that $\tilde{L}B = {u_n\left[1, 0, \ldots,0\right]}^T$.  Now, we can rewrite the iterative formula as

\noindent{\bf Implicit SDG Scheme:}
\begin{align}  
    \label{eq:SDG-Im-0}
    u^{k+1}_{n,0} & = u_n + \frac{\Delta t_n}{2}\omega_0\left(f(u^{k+1}_{n,0}) - f(u^{k}_{n,0})\right)  + \frac{\Delta t_n}{2} \sum\limits_{j=0}^p \tilde{l}_{0j} \omega_j f(u^{k}_{n,j}) \\
    u^{k+1}_{n,m+1} & = u^{k+1}_{n,m} + \frac{\Delta t_n}{2}\omega_{m+1}\left(f(u^{k+1}_{n,m+1}) - f(u^{k}_{n,m+1})\right) \notag \\
    \label{eq:SDG-Im-m}
    &\qquad\qquad + \frac{\Delta t_n}{2} \sum\limits_{j=0}^p \tilde{l}_{m+1,j} \omega_j f(u^{k}_{n,j}), \qquad 0 \leq m \leq p-1. 
\end{align}
where $U^0$ can be obtained, for example, by using the implicit Euler method.

\begin{rem}
  We note that matrix $\tilde{L}$ only depends on the polynomial degree $p$, which means it only needs to be computed once during the solving process. 
\end{rem}

Similar to the implicit scheme~\eqref{eq:SDG-Im-m}, we can derive an explicit version 

\noindent{\bf Explicit SDG Scheme:}
\begin{align} 
    \label{eq:SDG-Ex-0}
    u^{k+1}_{n,0} & =  u_n + \frac{\Delta t_n}{2}\sum\limits_{j=0}^p {\tilde{l}}_{0j} \omega_j  f(u^{k}_{n,j}) \\
    u^{k+1}_{n,m+1} & = u^{k+1}_{n,m} + \frac{\Delta t_n}{2}\omega_{m}\left(f(u^{k+1}_{n,m}) - f(u^{k}_{n,m})\right) \notag \\
    \label{eq:SDG-Ex-m}
    &\qquad\qquad + \frac{\Delta t_n}{2} \sum\limits_{j=0}^p \tilde{l}_{m+1,j}\omega_j f(u^{k}_{n,j}), \qquad 0 \leq m \leq p-1, 
\end{align} 
where $U^0$ can be obtained, for example, by using the explicit Euler method.

\begin{rem}
  We note that schemes~\eqref{eq:SDG-Ex-m} and~\eqref{eq:SDG-Im-m} are very similar to the explicit and implicit SDC schemes~\cite{Dutt:2000}.  For example, the implicit SDC scheme has form 
\[
  u^{k+1}_{n,m+1} = u^{k+1}_{n,m} + \Delta t_{n,m} \left(f(u^{k+1}_{n,m+1}) - f(u^{k}_{n,m+1})\right) + \frac{\Delta t_n}{2}\sum\limits_{j=0}^p q_{m+1,j} f(u^{k}_{n,j}), \qquad 0 \leq m \leq p-1.
\]
The computational cost of one SDG iteration is the same as for one SDC iteration with the same collocation nodes.  Therefore, for the complexity of the SDG methods, one can simply refer to the complexity of SDC methods~\cite{Dutt:2000, Minion:2003}.  However, we emphasize that the error estimate for SDC methods is $\min\{p+1, K+1\}$~\cite{Xia:2007} instead of $\min\{2p+1, K+1\}$ in Theorem~\ref{thm:SDG}.

Due to the similarity between the new proposed scheme and the SDC scheme, we simply name the new scheme as ``SDG'', where letter ``S'' can refer either ``spectral'' (like SDC) or ``superconvergence''.
\end{rem}

\begin{rem}
  Instead of deriving the SDG schemes from the previous algebra formulation, the same formula can also be obtained by considering a correction method based on the weak Galerkin form.  

  Denote $u^k \in V_h^p$ the $k$-th approximation to the DG solution $u_h$, the defect equation is given by
\[
  \delta^k = u_h - u^k \in V_h^p.
\]
Substituting $u_h = u^k + \delta_h$ into the weak form~\eqref{eq:weak}, we have
\begin{align*} 
    &-\int_{t_n}^{t_{n+1}} \delta^k {(v_h)}_t \,dt + \delta^k(t_{n+1}^{-})v_h(t_{n+1}^{-}) -  \int_{t_n}^{t_{n+1}} \left(f(t,u^k + \delta^k) - f(t,u^k)\right)v_h \,dt \\
    =\, & - \left( -\int_{t_n}^{t_{n+1}} u^k {(v_h)}_t \,dt + u^k(t_{n+1}^{-})v_h(t_{n+1}^{-}) - u_h(t_{n}^{- })v_h(t_{n}^{+}) - \int_{t_n}^{t_{n+1}}f(t, u^k)v_h \,dt\right).
\end{align*}
Since $u^k$ is an approximation of $u_h$, we have, for the right hand side
\[
    -\int_{t_n}^{t_{n+1}} u^k {(v_h)}_t \,dt + u^k(t_{n+1}^{-})v_h(t_{n+1}^{-}) - u_h(t_{n}^{- })v_h(t_{n}^{+}) - \int_{t_n}^{t_{n+1}}f(t, u^k)v_h \,dt \approx 0. 
\]
Therefore, we can approximate the defect $\delta^k$ by piecewise constant approximations
\[
  \delta^k(t) \approx \delta_h(t) = \sum\limits_{m=0}^{p-1} \delta_{n,m}\chi[t_{n,m},t_{n,m+1}], \quad t \in [t_n, t_{n+1}],
\]
which satisfies the weak form
\[
    -\int_{t_n}^{t_{n+1}} \delta_h {(v_h)}_t \,dt + \delta_h(t_{n+1}^{-})v_h(t_{n+1}^{-}) =   \int_{t_n}^{t_{n+1}} \left(f(t,u^k + \delta_h) - f(t,u^k)\right)v_h \,dt \qquad \forall v_h \in V_h^p.
\]
Finally, we update the new approximation $u^{k+1} = u^{k} + \delta_h$.  
\end{rem}

\begin{lmm}\label{lmm:SDG} (Local Truncation Error.)
  Denote $U^K = {[u^K_{n,0},\ldots,u^K_{n,p}]}^T$ the $K$-th iteration of the explicit or the implicit SDG scheme and $u$ the exact solution for the IVP~\eqref{eq:ivp}, then we have 
  \[
    u(t_{n,m}) - u^K_{n,m} = \mathcal{O}\left(h^{\min\{p+2, K+2\}}\right), \quad 0 \leq m \leq p-1.
  \]
  and at the end point $t_{n+1} = t_{n,p}$,
  \[
    u(t_{n+1}) - u^K_{n,p} = \mathcal{O}\left(h^{\min\{2p+2, K+2\}}\right),
  \]
  where $h = \max\limits_{n} \Delta t_n$.
\end{lmm}

\begin{proof}
  The proof is similar as Lemma~\ref{lmm:DG-Ex2}, as all the schemes are sourced from the preconditioned scheme~\eqref{eq:preDG}
\end{proof}

As the one-step time stepping methods, from Lemma~\ref{lmm:SDG}, we have 
\begin{thm}\label{thm:SDG} (Global Error.)
  The explicit and implicit SDG methods with $K$ iterations are $\min\{2p+1, K+1\}$ order accurate methods for the IVP~\eqref{eq:ivp}. 
\end{thm}

When the right hand side $f(t,u)$ of the IVP~\eqref{eq:ivp} can be split into a non-stiff term $f_N(t,u)$ and a stiff term $f_S(t,u)$, we have 
\begin{align}   
    \label{eq:ivp2}
        & u_t  = f(t, u(t)) = f_N(t, u(t)) + f_S(t, u(t)), \,\, t \in [0, T] \\
        & u(0) = u_0 \notag.
\end{align}
By combining the explicit and implicit SDG schemes, we can easily construct a semi-implicit or say implicit-explicit (IMEX) scheme as 

\noindent{\bf Semi-Implicit SDG Scheme:}
\begin{align}
    \label{eq:SDG-SI-0}
    u^{k+1}_{n,0} & = u_n + \frac{\Delta t_n}{2}\omega_0\left(f_S(u^{k+1}_{n,0}) - f_S(u^{k}_{n,0})\right)  + \frac{\Delta t_n}{2} \sum\limits_{j=0}^p \tilde{l}_{0j} \omega_j f(u^{k}_{n,j}) \\
    u^{k+1}_{n,m+1} & = u^{k+1}_{n,m} + \frac{\Delta t_n}{2}\omega_{m}\left(f_N(u^{k+1}_{n,m}) - f_N(u^{k}_{n,m})\right)+\frac{\Delta t_n}{2}\omega_{m+1}\left(f_S(u^{k+1}_{n,m+1}) - f_S(u^{k}_{n,m+1})\right)\notag \\
    \label{eq:SDG-SI-m}
    &\qquad\qquad\qquad + \frac{\Delta t_n}{2} \sum\limits_{j=0}^p \tilde{l}_{m+1,j} \omega_j f(u^{k}_{n,j}), \qquad 0 \leq m \leq p-1. 
\end{align}

To simplify the writing, we refer the explicit, implicit and semi-implicit SDG methods using the polynomial of degree $p$ and $K$ iterative steps to $ExSDG_p^K$, $ImSDG_p^K$, and $SISDG_p^K$.

\subsubsection{Example}
To illustrate how the SDG methods work, we apply the SDG methods to the classic Dahlquist’s test problem
\begin{equation} 
  \label{eq:Dahlquist}
    \left\{\begin{array}{l} 
      u_t = \lambda u(t), \\
      u(0) = 1,
    \end{array}\right.
\end{equation}
until time $t = 1$ with one step ($\Delta t = 1$).  To avoid difference caused by different initialization methods (explicit or implicit Euler method), we simply initialize the the iteration with the constant values $u^0_m = u(0),\,m = 0,\dots,p$.  Note, since the initial step has convergence rate of $\mathcal{O}(1)$, theoretically, we need $2p+1$ iterations to achieve the convergence rate of $\mathcal{O}(h^{2p+1})$.  In Figure~\ref{fig:Dahlquists1}, we compared global errors of using the simple explicit scheme~\eqref{eq:DG-Ex}, the explicit and implicit SDG methods.  One can see that as Theorem~\ref{thm:SDG} suggested, each iteration enhance the accuracy by one order, and finally all three methods convergent to the standard DG approximation.  We also note that the naive explicit scheme~\eqref{eq:DG-Ex} has worse performance compared to the two SDG schemes, although it has the simplest formula.      

\begin{figure}[!ht] 
    \centering
    \begin{minipage}[c]{0.48\textwidth}
      \centerline{$\qquad p = 5$}
    \includegraphics[width=1.\textwidth]{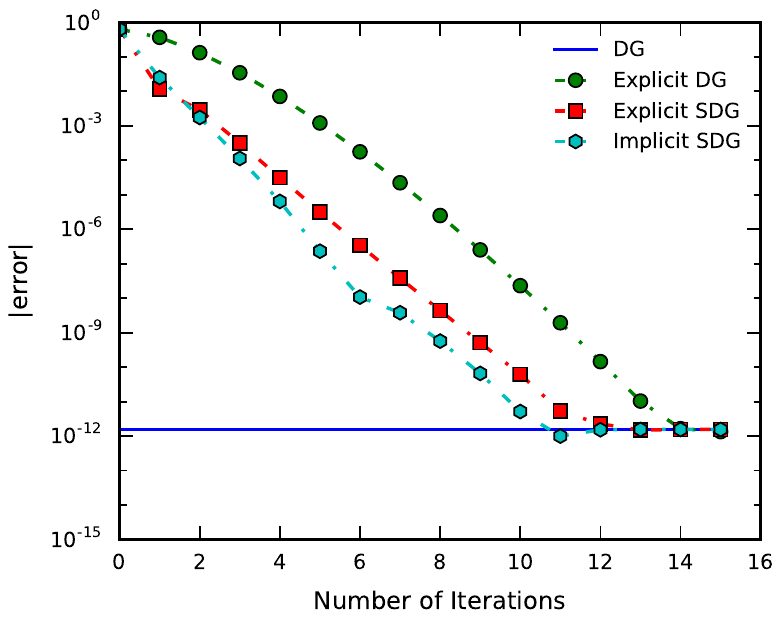}
    \end{minipage}
    \begin{minipage}[c]{0.48\textwidth}
      \centerline{$\qquad p = 6$}
    \includegraphics[width=1.\textwidth]{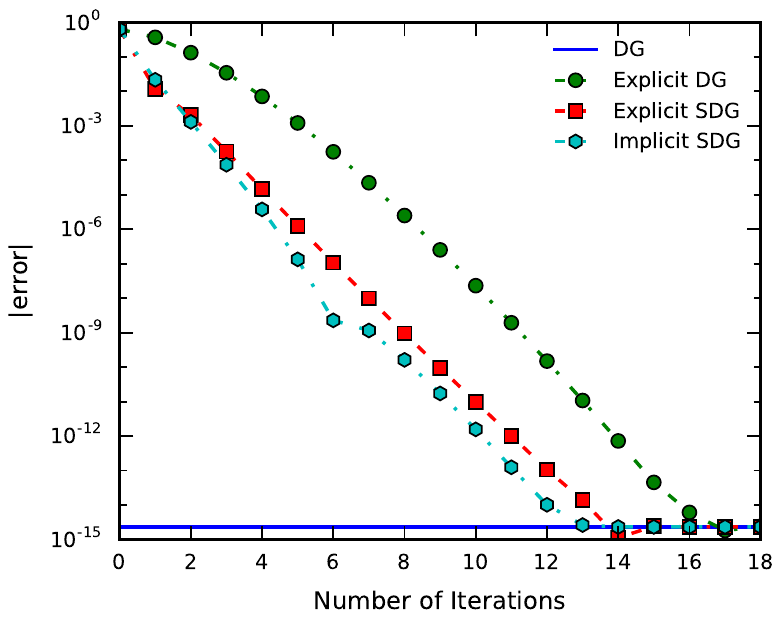}
    \end{minipage}
    
    \begin{center}
    \caption{\label{fig:Dahlquists1}
    Errors for the Dahlquist's test problem~\eqref{eq:Dahlquist} with $\lambda = -1$ and $\Delta t = 1$ computed  with using the $ExDG_p^K$, $ExSDG_p^K$, and $ImSDG_p^K$ methods for $K$ ranging from $0$ to $3p$.  
}
     \end{center}
\end{figure}

    

\subsection{Stability Property of SDG schemes}
Usually, we are concerned about two critical characteristics of a time stepping scheme. One of them is the order of accuracy, which we have already proved for the SDG schemes in Theorem~\ref{thm:SDG}.  In this section, we look to another important characteristic of the SDG schemes: stability. 

The stability of a numerical method is general analyzed by applying it to the Dahlquist's test problem~\eqref{eq:Dahlquist} with $t \in [0,1]$.  For $\lambda \in \mathbb{C}$, we have 
\[
    u(1) = Am(\lambda)u(0),
\]
where $Am(\lambda)$ is defined as the amplification factor.  The stability region of a numerical scheme for the equation~\eqref{eq:Dahlquist} is defined as the subset of the complex plane $\mathbb{C}$ consisting of all $\lambda$ such that $Am(\lambda) \leq 1$.  

First, we compute the stability region for the $ExSDG_p^K$ method for several choice of $p$ with $K = p$ and $K = 2p$, see Figure~\ref{fig:StabilityExSDG}.  Here, for comparison, we also provide the stability region for the popular explicit third and fourth order Runge-Kutta methods in Figure~\ref{fig:StabilityERK}.  In Figure~\ref{fig:StabilityExSDG}, we can see that the size of stability regions grows with the polynomial degree $p$, also the stability regions of explicit SDG methods are clearly larger compared to explicit Runge-Kutta methods in Figure~\ref{fig:StabilityERK}.  This result suggests that the $ExSDG_p^K$ method is suitable especially for non-stiff and little stiff problems.  We also note that if one only requires the regular convergence rate $p+1$ ($K = p$), the stability region is slightly larger than the superconvergence case ($K = 2p$). 

\begin{figure}[!ht]  
    \centering
    \begin{flushleft}
      Regular Convergence ($K = p$)
    \end{flushleft}
    \begin{minipage}[c]{0.48\textwidth}
      \centerline{$\quad\,\, p = 4$}
    \includegraphics[width=1.\textwidth]{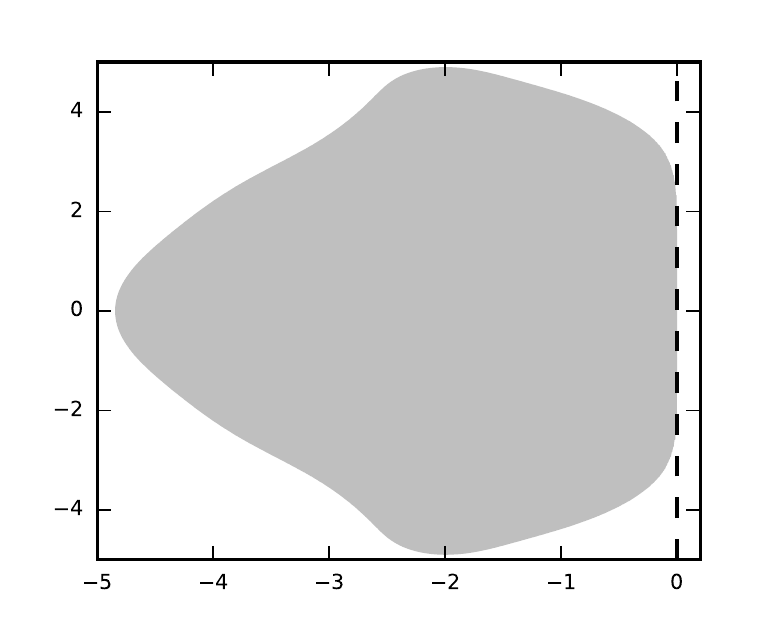}
    \end{minipage}
    \begin{minipage}[c]{0.48\textwidth}
      \centerline{$\quad\,\, p = 8$}
        \includegraphics[width=1.\textwidth]{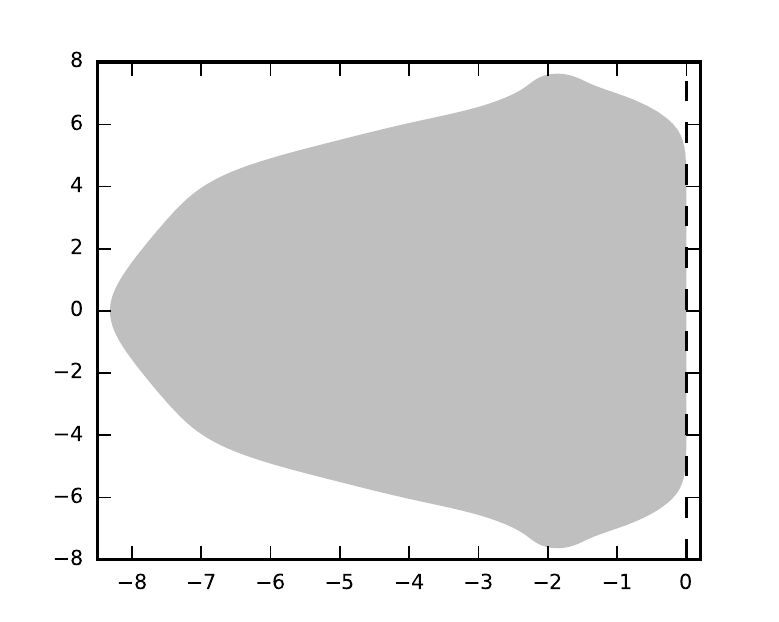}
    \end{minipage}

    \begin{flushleft}
      Superconvergence ($K = 2p$)
    \end{flushleft}
    \begin{minipage}[c]{0.48\textwidth}
      \centerline{$\quad\,\, p = 4$}
    \includegraphics[width=1.\textwidth]{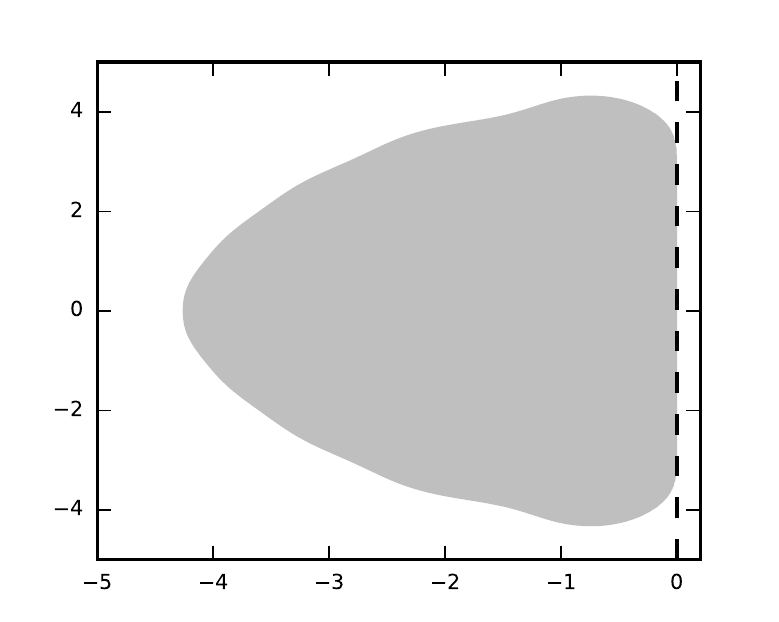}
    \end{minipage}
    \begin{minipage}[c]{0.48\textwidth}
      \centerline{$\quad\,\, p = 8$}
    \includegraphics[width=1.\textwidth]{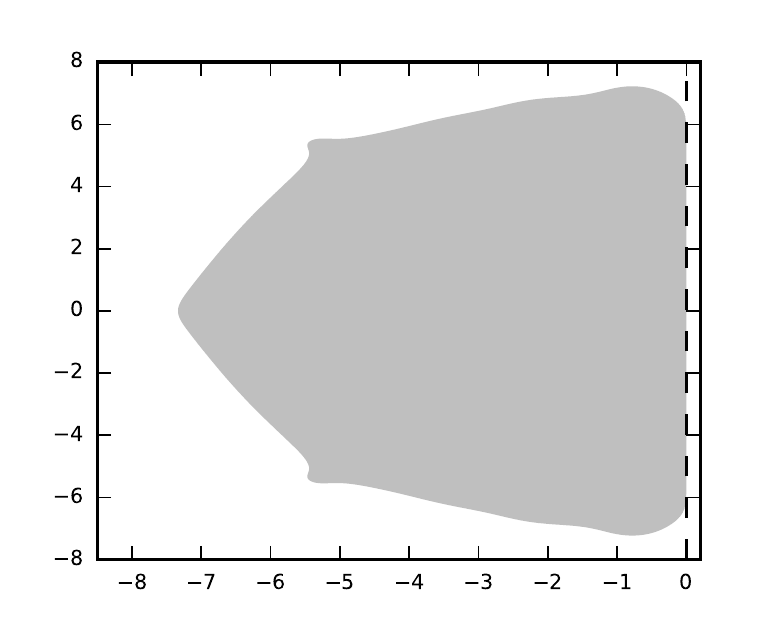}
    \end{minipage}
    
    \begin{center}
    \caption{\label{fig:StabilityExSDG}
    Stability regions (grey) for $ExSDG_p^K$ with $p = 4$ and $p = 8$.  First row: regular convergence ($K = p$); second row: superconvergence ($K = 2p$) 
}
     \end{center}
\end{figure}

\begin{figure}[!ht]  
    \centering
    \begin{minipage}[c]{0.48\textwidth}
      \centerline{$\quad\,\, \text{Runge-Kutta 3}$}
    \includegraphics[width=1.\textwidth]{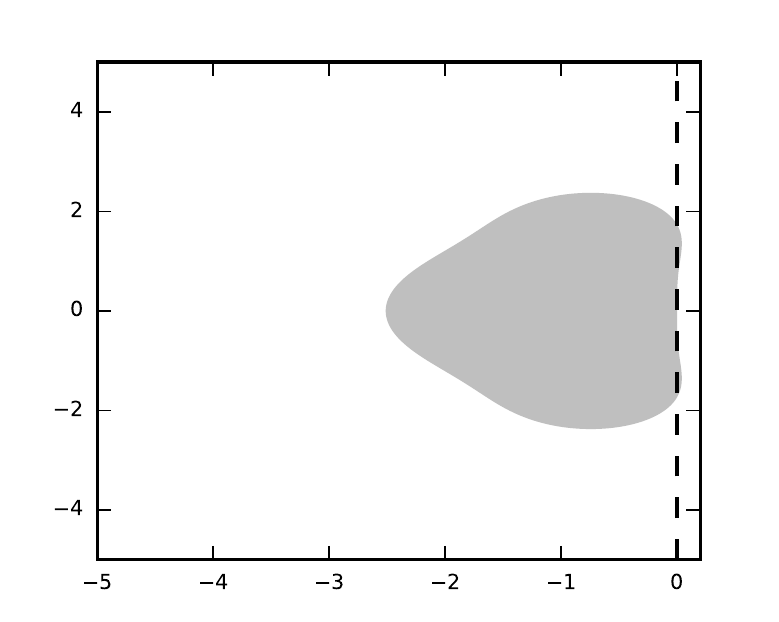}
    \end{minipage}
    \begin{minipage}[c]{0.48\textwidth}
      \centerline{$\quad\,\, \text{Runge-Kutta 4}$}
        \includegraphics[width=1.\textwidth]{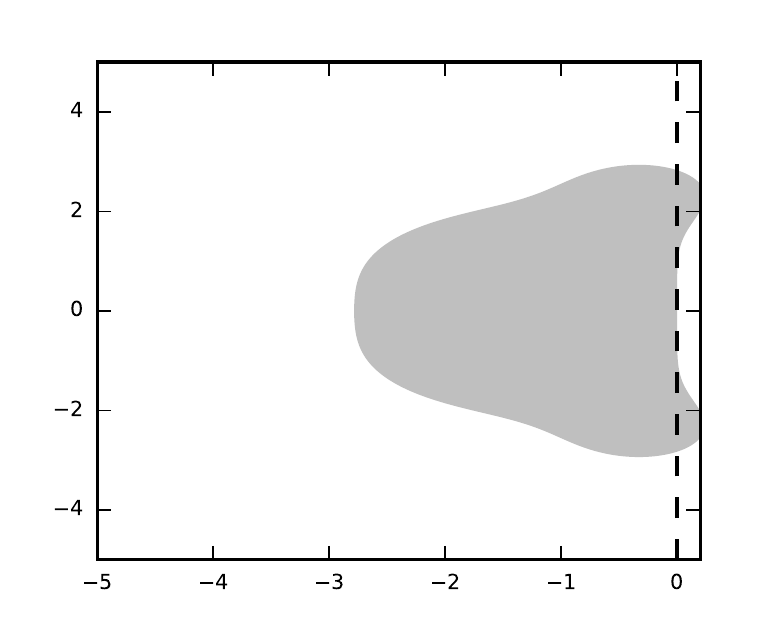}
    \end{minipage}

    \begin{center}
    \caption{\label{fig:StabilityERK}
    Stability regions (grey) for explicit third and fourth order Runge-Kutta methods.
}
     \end{center}
\end{figure}

Regarding stiff problems, we need solve them with the $ImSDG_p^K$ methods.  In Figure~\ref{fig:StabilityImSDG}, we see that for both regular convergence ($K = p$) and superconvergence ($K = 2p$) cases, whenever $\operatorname{Re}(\lambda) < 0$ the scheme is stable. Based on our test, for $0 \leq p \leq 9$, the $ImSDG_p^K$ methods are always $A$-stable.  Although we do not have a proof yet, we expect that all of the $ImSDG_p^K$ methods are $A$-stable.

\begin{figure}[!ht] 
    \centering
    \begin{flushleft}
      Regular Convergence ($K = p$)
    \end{flushleft}
    \begin{minipage}[c]{0.48\textwidth}
      \centerline{$\quad\,\, p = 4$}
    \includegraphics[width=1.\textwidth]{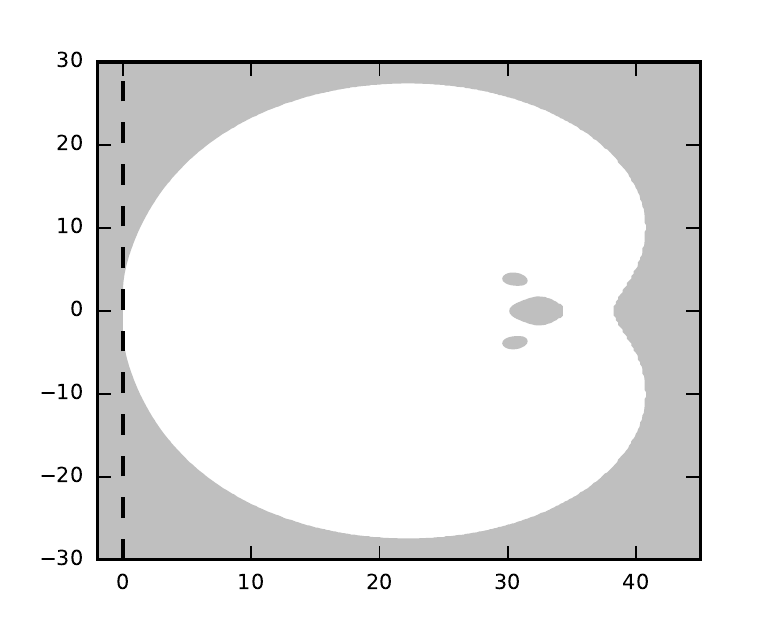}
    \end{minipage}
    \begin{minipage}[c]{0.48\textwidth}
      \centerline{$\quad\,\, p = 8$}
    \includegraphics[width=1.\textwidth]{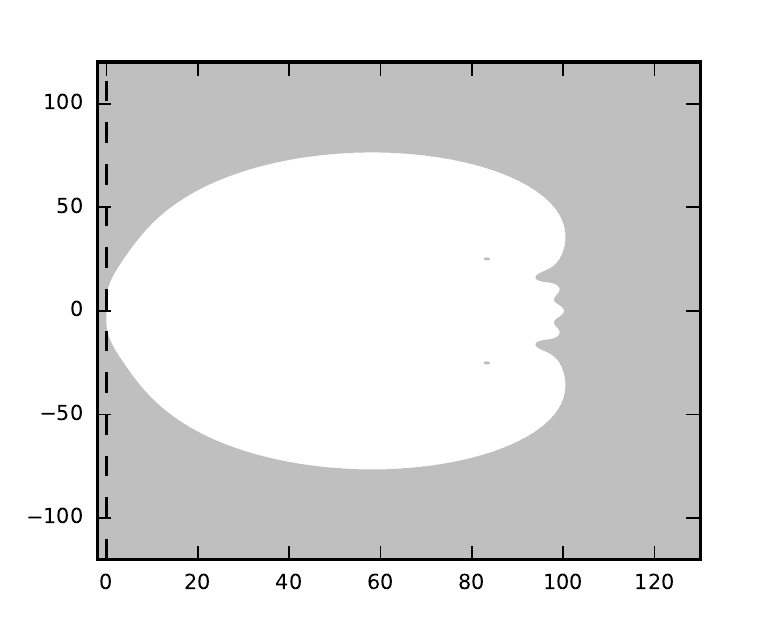}
    \end{minipage}

    \begin{flushleft}
      Superconvergence ($K = 2p$)
    \end{flushleft}
    \begin{minipage}[c]{0.48\textwidth}
      \centerline{$\quad\,\, p = 4$}
    \includegraphics[width=1.\textwidth]{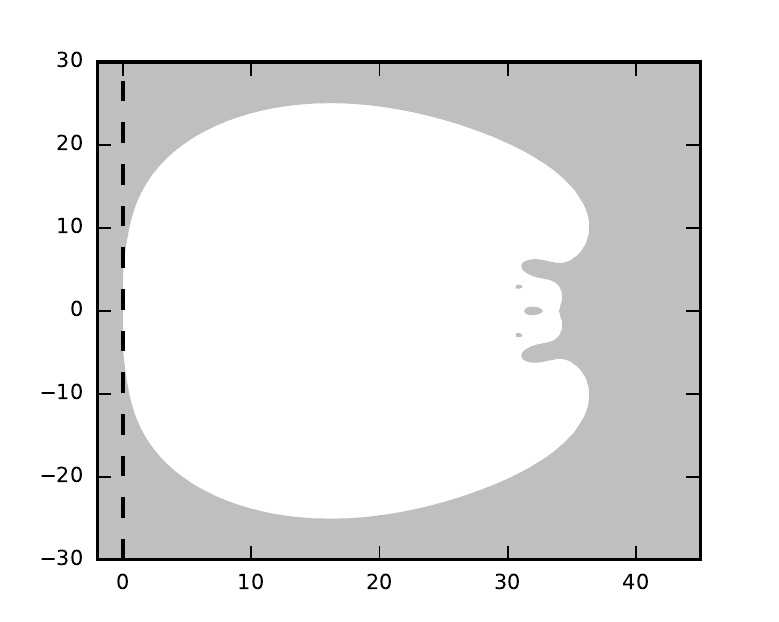}
    \end{minipage}
    \begin{minipage}[c]{0.48\textwidth}
      \centerline{$\quad\,\, p = 8$}
    \includegraphics[width=1.\textwidth]{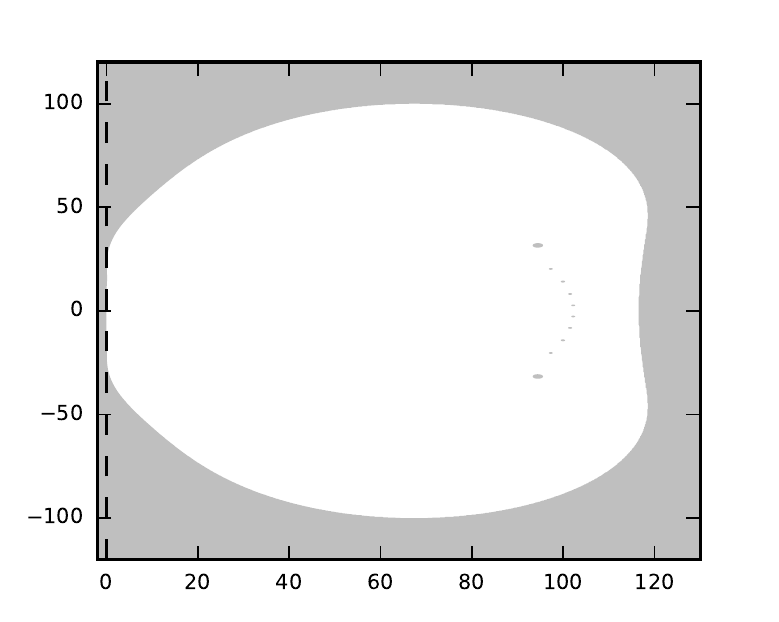}
    \end{minipage}
    
    \begin{center}
    \caption{\label{fig:StabilityImSDG}
    Stability regions (grey) for $ImSDG_p^K$ with $p = 4$ and $p = 8$.  First row: regular convergence ($K = p$); second row: superconvergence ($K = 2p$) 
}
     \end{center}
\end{figure}

For very stiff problems, one may want to use the so-called $L$-stable schemes, see~\cite{Dutt:2000} such that 
\[
    Am(\lambda) \to 0, \quad \text{if} \quad \operatorname{Re}(\lambda) \to -\infty. 
\]
One easiest way to construct a $L$-stable SDG scheme is to introduce a $\theta$ method $ImSDG_p^K(\theta)$ based on $ImSDG_p^K$: 

\noindent{\bf Implicit SDG $\theta$-Scheme ($ImSDG_p^K(\theta)$):}
\begin{align}  
    \label{eq:SDGtheta-Im-0}
    u^{K}_{n,0} & = u_n + \theta\frac{\Delta t_n}{2}\omega_0\left(f(u^{K}_{n,0}) - f(u^{K-1}_{n,0})\right)  + \frac{\Delta t_n}{2} \sum\limits_{j=0}^p \tilde{l}_{0j} \omega_j f(u^{K-1}_{n,j})  \\
    u^{K}_{n,m+1} & = u^{K}_{n,m} + \theta\frac{\Delta t_n}{2}\omega_{m+1}\left(f(u^{K}_{n,m+1}) - f(u^{K-1}_{n,m+1})\right) \notag \\
    \label{eq:SDGtheta-Im-m}
    &\qquad\qquad + \frac{\Delta t_n}{2} \sum\limits_{j=0}^p \tilde{l}_{m+1,j} \omega_j f(u^{K-1}_{n,j}), \qquad 0 \leq m \leq p-1. 
\end{align}
where the $U^0,\dots,U^{K-1}$ are obtained by using the $ImSDG_p^K$ method.

Then, same as~\cite{Dutt:2000, Xia:2007}, a $L$-stable scheme can be constructed by combining the $ImSDG_p^K(\theta)$ and $ImSDG_p^K$ schemes.

\section{Multilevel and Adaptive Strategy}
\subsection{A Multilevel SDG Methods} 
In the previous section, we introduced a new class of iterative schemes based on the DG time stepping methods.  Furthermore, in this section, we develop a multigrid method for the time domain which is using our SDG schemes as smoother.  Due to the nature of the DG approximation space~\eqref{eq:space},  the coarse ``grids'' can be constructed by reducing the degree $p$ of the approximating basis functions ($p$-multigrid).  A similar multilevel approach with the SDC methods can be found in~\cite{Emmett:2012,Speck:2015}.

As described in~\eqref{eq:preDG}, the SDG schemes are iterative methods, which are applied to the DG weak formula~\eqref{eq:weak}. For convenience we rewrite \eqref{eq:non2} as  
\[
  U + \Delta t L^{-1}F(U) + L^{-1}B = 0,
\]
with $\Delta t = \frac{\Delta t_n}{2}$ and $F = F_\omega$.  

In order to derive the multilevel algorithm, in this paper, we use the full approximation scheme (FAS) to treat the nonlinearity directly.  For details of multigird methods and FAS correction technique, see \cite{Briggs:2000}.  We define levels $\ell = 1,\dots,L$, where $\ell=1$ is the finest level $V^p_h$.  Also, we define the operator $A_\ell$ on the $\ell$-th level, as
\[
  A_\ell(U_\ell) \equiv U_\ell + \Delta t L^{-1}_\ell F_{\ell} (U_\ell). 
\]
Then, the FAS correction for level $\ell+1$ is given by
\[
  \tau_{\ell+1} = A_{\ell+1}(I_\ell^{\ell+1}(U_\ell)) - I_\ell^{\ell+1}A_\ell(U_\ell) 
  = \Delta t \left(L^{-1}_{\ell+1} F_{\ell+1}(I_\ell^{\ell+1}U_\ell) - I_\ell^{\ell+1}L^{-1}_\ell F_{\ell}(U_\ell)\right).
\]
However, if on level $\ell$ the equation is already corrected by $\tau_\ell$ with 
\[
  A_\ell(U_\ell) = U_\ell + \Delta t L^{-1}_\ell F_{\ell} - \tau_\ell,
\]
then 
\begin{align*}
  \tau_{\ell+1} & = A_{\ell+1}(I_\ell^{\ell+1}(U_\ell)) - I_\ell^{\ell+1}A_\ell(U_\ell)  \\
  & = \Delta t \left(L^{-1}_{\ell+1} F_{\ell+1}(I_\ell^{\ell+1}U_\ell) - I_\ell^{\ell+1}L^{-1}_\ell  F_{\ell}(U_\ell)\right) + I_\ell^{\ell+1}\tau_\ell.
\end{align*}
On level $\ell$, the corrected weak formula is 
\[
  U_\ell + \Delta t L^{-1}_\ell F_\ell(U) + L^{-1}_\ell B_\ell - \tau_\ell = 0,
\]
which can be solved by the SDG methods in the same way.

Algorithm~\ref{algorithm} describes one multilevel SDG iteration.  It is worth noting that in Algorithm~\ref{algorithm} the word ``$\text{SDG\_Sweep}$'' is used for one SDG iteration, which can originate from the explicit, implicit, or semi-implicit scheme.  Also, the operators $L^{-1}_l$ can be precomputed and stored. 

\begin{algorithm}
  \caption{Multilevel SDG} \label{algorithm}
  \DontPrintSemicolon
  \KwData{Upwind flux $U_{1,0}$ from the previous time step, values $U^k_1$ and $F_1^k = F_1(U_1^k)$ from the previous iteration, on the finest level $\ell = 1$.}
  \KwResult{Solution $U^{k+1}_1$.}
  \hfill\\
  \tcp{Perform a fine level sweep using SDG}
  $U_1^{k+1}, F_1^{k+1} \longleftarrow \text{SDG\_Sweep}(U_{1,0}, U^{k}_1, F^k_1)$\; 
  \hfill\\
  \tcp{Cycle from fine to coarse}
  \For{$\ell = 1 \ldots L-1$}{
    $U^k_{\ell+1} \longleftarrow I_\ell^{\ell+1}(U^{k+1}_\ell)$ \tcp*[r]{restrict}
    $F^k_{\ell+1} \longleftarrow F_{\ell+1}(U^k_{\ell+1})$ \tcp*[r]{evaluation of $F$}
    \hfill\\
    \tcp{FAS correction and sweep}
    $\tau_{\ell+1} \longleftarrow \text{FAS}(F^{k+1}_\ell, F^k_{\ell+1}, \tau_\ell)$\;
    $U^{k+1}_{\ell+1}, F^{k+1}_{\ell+1} \longleftarrow \text{SDG\_Sweep}(U^k_{1,0^{-}}, U^k_{\ell+1}, F_{\ell+1}^k, \tau_{\ell+1})$\;
  }
  \hfill\\
  \tcp{Cycle from coarse to fine}
  \For{$\ell = L-1 \ldots 2$}{
    $U^{k+1}_\ell \longleftarrow U^{k+1}_\ell + I_{\ell+1}^{\ell}(U^{k+1}_{\ell+1} - U^k_{\ell+1})$ \tcp*[r]{interpolate coarse level correction}
    $F^{k+1}_\ell \longleftarrow F_{\ell}(U^k_{\ell})$ \tcp*[r]{evaluation of $F$}
    $U^{k+1}_{\ell}, F^{k+1}_\ell \longleftarrow \text{SDG\_Sweep}(U^k_{1,0^{-}}, U^{k+1}_\ell, F^{k+1}_\ell, \tau_\ell)$\;
  }
  \hfill\\
  \tcp{Return to finest level}
  $U^{k+1}_1 \longleftarrow U^{k+1}_1 + I_{2}^{1}(U^{k+1}_{2} - \tilde{U}^k_2)$\;
  $F^{k+1}_1 \longleftarrow F_{1}(U^{k+1}_{1})$\;
\end{algorithm}

To compare the Multilevel SDG method with one-level SDG, we apply them for the Dahlquist's test problem~\eqref{eq:Dahlquist} with (large) negative eigenvalues, i.e.\ the stiff case.  For $\lambda = -10$, we show the results of using two-level and three-level implicit SDG methods in Figures~\ref{fig:MLSDG2} and~\ref{fig:MLSDG3}, respectively.  Compared to the one-level version, both the two- and three-level methods can reduce the number of iterations required to converge to the standard DG approximation.  We note that the convergence speed is also affected by the choices of different orders ($p_i$) and the time step size ($\Delta t$).  Since the main purpose of this paper is to introduce a new class of time stepping schemes -- the SDG schemes, we note that the further studies and applications of the multilevel technique will be studied in a forthcoming paper. 

\begin{rem}
  Multilevel SDG method can be combined with the method of lines to solve PDEs.  Moreover, it can be coupled with a spatial multigrid to construct a space-time multigrid framework for high-order methods.  Furthermore, due to its structure, the multilevel method is also a competitive and promising candidate for parallel in time algorithms, such as Parareal or PFASST~\cite{Emmett:2012}.  
\end{rem}

\begin{figure}[!ht] 
    \centering
    \begin{minipage}[c]{0.48\textwidth}
      \centerline{$\qquad\,\, p = 6,\,\,\Delta t = 0.2$}
    \includegraphics[width=1.\textwidth]{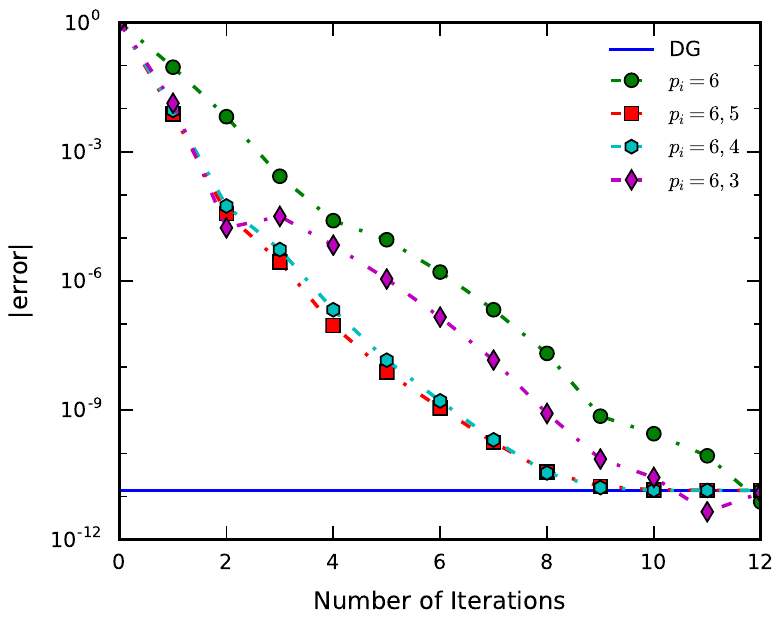}
    \end{minipage}
    \begin{minipage}[c]{0.48\textwidth}
      \centerline{$\qquad\,\, p = 9,\,\,\Delta t = 0.5$}
    \includegraphics[width=1.\textwidth]{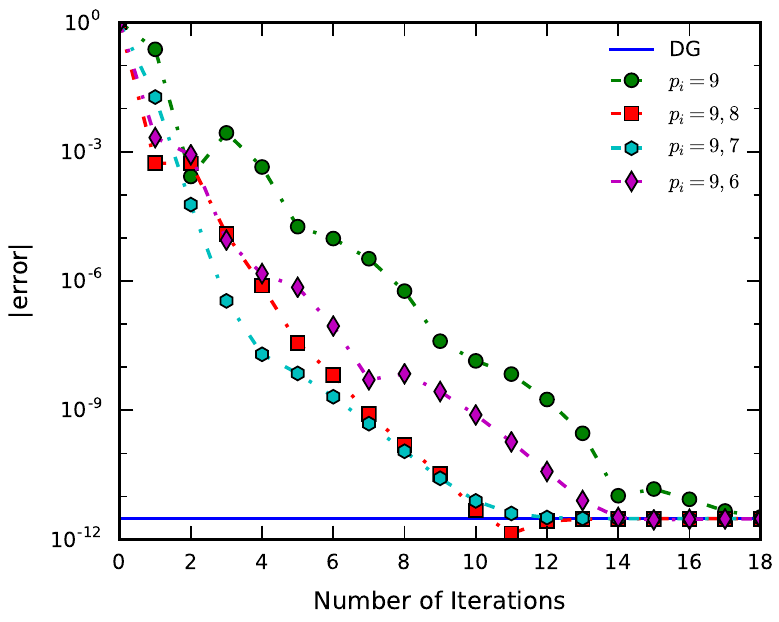}
    \end{minipage}
    
    \begin{center}
    \caption{\label{fig:MLSDG2}
    Errors for Dahlquist's test problem~\eqref{eq:Dahlquist} with $\lambda = -10$ computed using the two-level implicit SDG methods. The result of using the one-level implicit SDG method are marked with green circles.   
}
     \end{center}
\end{figure}

\begin{figure}[!ht] 
    \centering
    \begin{minipage}[c]{0.48\textwidth}
      \centerline{$\qquad\,\, p = 6,\,\,\Delta t = 0.2$}
    \includegraphics[width=1.\textwidth]{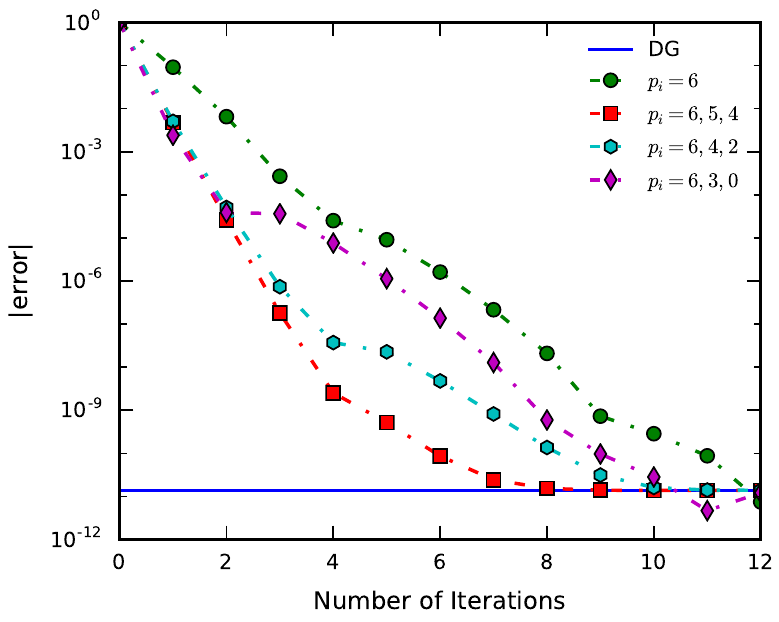}
    \end{minipage}
    \begin{minipage}[c]{0.48\textwidth}
      \centerline{$\qquad\,\, p = 9,\,\,\Delta t = 0.5$}
    \includegraphics[width=1.\textwidth]{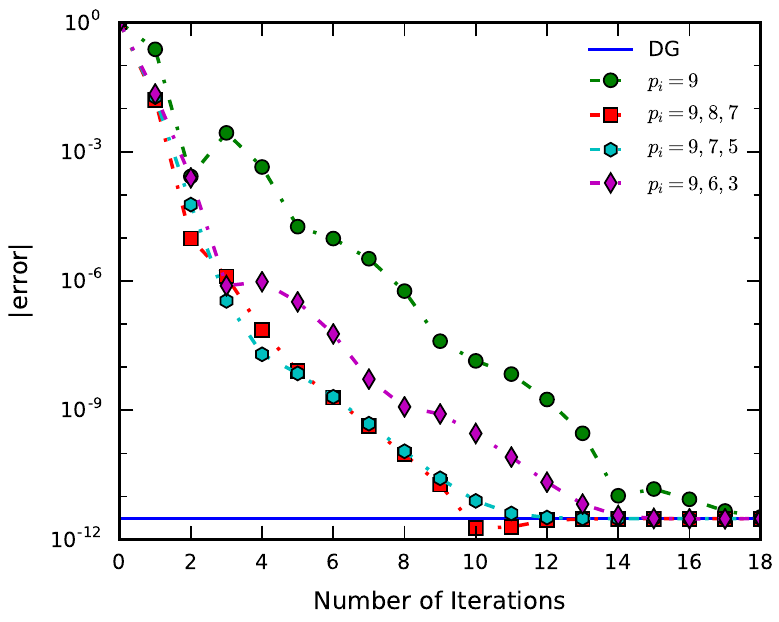}
    \end{minipage}
    
    \begin{center}
    \caption{\label{fig:MLSDG3}
    Errors for Dahlquist's test problem~\eqref{eq:Dahlquist} with $\lambda = -10$ computed using the three-level implicit SDG methods.  The result of using the one-level implicit SDG method are marked with green circles.
}
     \end{center}
\end{figure}

\subsection{Adaptive Strategy}
Adaptive step size control plays an important role in practical applications.  Our newly proposed SDG method, which essentially is a one-step scheme and an ideal candidate for adaptive implementation.  The adaptive strategy can be based on established techniques, such as local grid refinement ($h$-refinement) or the local choice of the order ($p$-adaptivity), see, e.g~\cite{Bottcher:1996}.   We note that the detailed studies and applications of adaptive implementation and accuracy control will be presented in a forthcoming paper.


\section{Numerical experiments}
In this section, we numerically validate the performance of the proposed SDG schemes with three examples.  

\subsection{Van der Pol Equation}
The first example is the Van der Pol Equation,
\begin{align} 
  \label{eq:van}
  u' & = v,\\
  v' & = (-u + (1 - {u}^2)v)/\epsilon,\notag
\end{align}
with initial values $u(0) = 2$ and $v(0) = −0.6666654321121172$ based on~\cite{Minion:2003,Zhao:2014}.  The Van der Pol equation is a well-known benchmark example for studying stiff ODE problems.  In this example, we solving it with the semi-implicit SDG methods:  the first equation is treated implicitly, and the second one explicitly.  


In this example, we only investigate the accuracy of the semi-implicit SDG scheme for solving the Van der Pol equation.  The maximum global errors are reported for a slightly stiff parameter $\epsilon = 10^{-1}$ in Table~\ref{table:s-1} and a stiff parameter $\epsilon = 10^{-3}$ in Table~\ref{table:s-3} with a short stopping time $T = 0.5$ are employed for this purpose.  Here, the reference solutions are computed using the high order semi-implicit SDG scheme ($p = 9$) on a very fine mesh ($\Delta t = 10^{-6}$).  For the slightly stiff case ($\epsilon = 10^{-1}$), Table~\ref{table:s-1}, the full accuracy order of $2p+1$ is observed.  For the stiff case ($\epsilon = 10^{-3}$), in Table~\ref{table:s-3} a slight order reduction from $2p+1$ is observed, but the error and accuracy order are still better than a standard method of order $p+1$.     


\begin{table}\centering 
\begin{tabular}{@{}cccccccc@{}}\toprule
  & & & \multicolumn{2}{c}{$u$} & \phantom{abc}& \multicolumn{2}{c}{$v$} \\ 
\cmidrule{4-5} \cmidrule{7-8}
Degree & $\Delta t$ & & Error & Order && Error & Order \\ 
\midrule
        & 2.50E-02 && 4.51E-14 &  --   && 3.91E-10 &  -- \\ 
$p = 3$ & 1.25E-02 && 3.47E-16 &  7.02 && 3.17E-12 & 6.95\\ 
        & 6.25E-03 && 2.69E-18 &  7.01 && 2.57E-14 & 6.95\\ 
\midrule
        & 2.50E-02 && 2.41E-14 &  --   && 3.13E-16 &  --  \\ 
$p = 4$ & 1.25E-02 && 4.84E-17 &  8.96 && 6.16E-19 & 8.99\\ 
        & 6.25E-03 && 9.77E-20 &  8.95 && 1.21E-21 & 8.99\\ 
\midrule
        & 2.50E-02 && 3.70E-20 &  --   && 8.15E-16 &  --  \\ 
$p = 5$ & 1.25E-02 && 2.03E-23 & 10.83 && 4.08E-19 & 10.96 \\ 
        & 6.25E-03 && 1.98E-26 & 10.00 && 2.03E-22 & 10.97 \\ 
\bottomrule
\end{tabular}
\caption{Convergence tests for the Van der Pol equation~\eqref{eq:van}. Here, $t = 0.5$ and $\epsilon = 10^{-1}$.}
\label{table:s-1}
\end{table}

\begin{table}\centering 
\begin{tabular}{@{}cccccccc@{}}\toprule
  & & & \multicolumn{2}{c}{$u$} & \phantom{abc}& \multicolumn{2}{c}{$v$} \\ 
\cmidrule{4-5} \cmidrule{7-8}
Degree & $\Delta t$ & & Error & Order && Error & Order \\ 
\midrule
        & 2.50E-03 && 3.57E-11 &  --   && 1.09E-06 &  -- \\ 
$p = 3$ & 1.25E-03 && 2.42E-13 &  7.21 && 8.08E-08 & 3.75\\ 
        & 6.25E-04 && 8.86E-15 &  4.77 && 1.79E-09 & 5.49\\ 
\midrule
        & 2.50E-03 && 5.33E-11 &  --   && 5.56E-13 &  --  \\ 
$p = 4$ & 1.25E-03 && 1.14E-12 &  5.55 && 1.47E-14 & 5.24\\ 
        & 6.25E-04 && 6.64E-15 &  7.42 && 5.95E-17 & 7.95\\ 
\midrule
        & 2.50E-03 && 4.44E-15 &  --   && 1.66E-08 &  --  \\ 
$p = 5$ & 1.25E-03 && 1.38E-17 &  8.33 && 9.87E-11 & 7.40 \\ 
        & 6.25E-04 && 6.38E-21 & 11.08 && 1.47E-13 & 9.39 \\ 
\bottomrule
\end{tabular}
\caption{Convergence tests for the Van der Pol equation~\eqref{eq:van}. Here, $t = 0.5$ and $\epsilon = 10^{-3}$.}
\label{table:s-3}
\end{table}

\subsection{A ``bad'' Example}
The second example is a ``bad'' example taken from~\cite{Iserles:2009},
\begin{equation}
  \label{eq:bad}
  y' = \ln3 \left( y - \lfloor y \rfloor - \frac{3}{2}\right),
\end{equation}
with $y(0) = 0$.  It is easy to verify that, the exact solution of this example~\eqref{eq:bad} is 
\[
  y(t) = -\lfloor y \rfloor + \frac{1}{2}\left(1 - 3^{t-\lfloor t \rfloor}\right),\qquad t \geq 0,
\]
where $\lfloor y \rfloor$ is the integer part of $y \in \mathbb{R}$.  Since the right-hand side function of~\eqref{eq:bad} does not satisfy the Lipschitz condition, the standard time-stepping methods do not perform well.  Here, we simply compare the new proposed explicit SDG method ($p=3$) with the classic $4$th order explicit Runge-Kutta method.  We observe that in Figure~\ref{fig:bad}, although the error decreases with $h$, the rate of global error decay for the Runge-Kutta method is just $\mathcal{O}(h)$.  The Runge-Kutta method performs badly as people may expect, and the source of error is the integer points where the function fails the Lipschitz condition (or smoothness requirement).  On the other sided, the SDG methods inherit the flexibility of dealing discontinuous from the DG methods; we can clearly see that it maintains the global error decay rate of $\mathcal{O}(h^7)$ as expected form the numerical analysis. 
\begin{rem}
  Here, we note that in this example, for both methods, we keep all discontinuities are located at the mesh points. This is important to maintain the optimal order of accuracy. In practice, we need to use an approximation of the location of discontinuities and adaptive strategies.  Otherwise, the rate of global error will deduce to $\mathcal{O}(h)$ for the SDG methods too. 
\end{rem}
\begin{figure}[!ht]
    \centering
    \begin{minipage}[c]{0.05\textwidth}
        \begin{sideways}
          \centerline{$\,$}
        \end{sideways}
    \end{minipage}
    \begin{minipage}[c]{0.45\textwidth}
      \centerline{$\qquad\quad \text{Runge-Kutta}$}
    \end{minipage}
    \begin{minipage}[c]{0.45\textwidth}
      \centerline{$\qquad\quad \text{SDG}$}
    \end{minipage}
    \begin{minipage}[c]{0.05\textwidth}
        \begin{sideways}
          \centerline{$\Delta t = 0.2$}
        \end{sideways}
    \end{minipage}
    \begin{minipage}[c]{0.45\textwidth}
    \includegraphics[width=1.\textwidth]{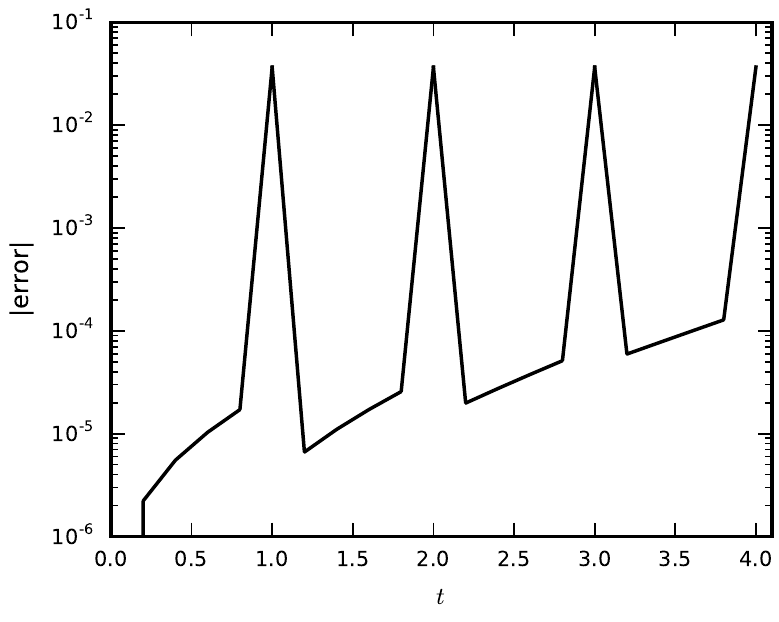}
    \end{minipage}
    \begin{minipage}[c]{0.45\textwidth}
    \includegraphics[width=1.\textwidth]{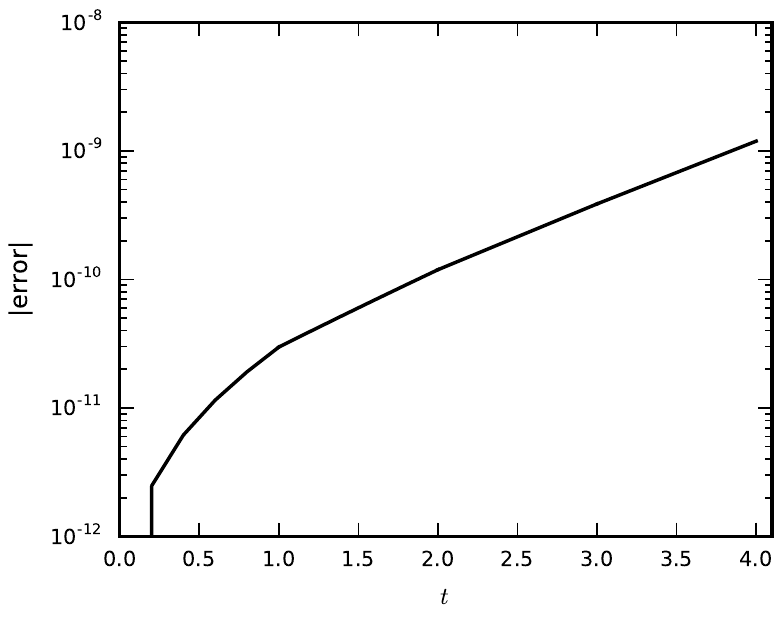}
    \end{minipage}
    
    \begin{minipage}[c]{0.05\textwidth}
        \begin{sideways}
          \centerline{$\Delta t = 0.1$}
        \end{sideways}
    \end{minipage}
    \begin{minipage}[c]{0.45\textwidth}
    \includegraphics[width=1.\textwidth]{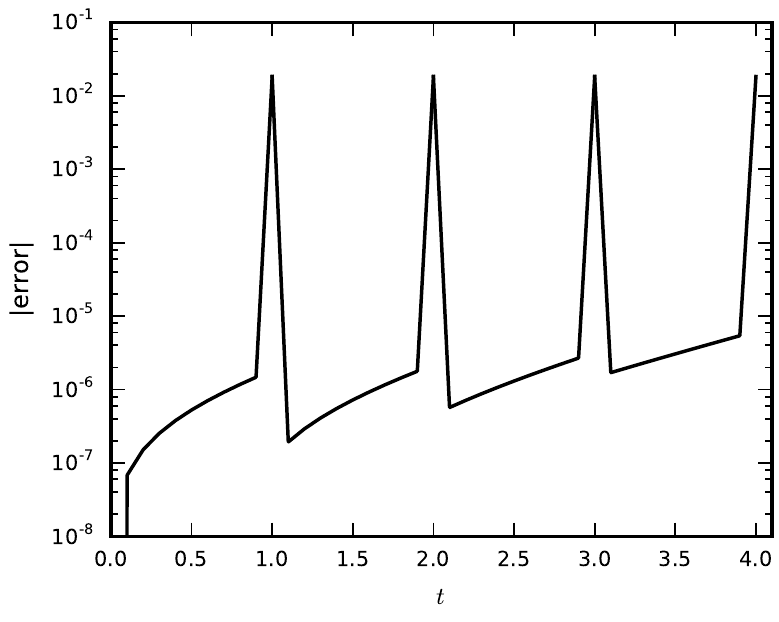}
    \end{minipage}
    \begin{minipage}[c]{0.45\textwidth}
    \includegraphics[width=1.\textwidth]{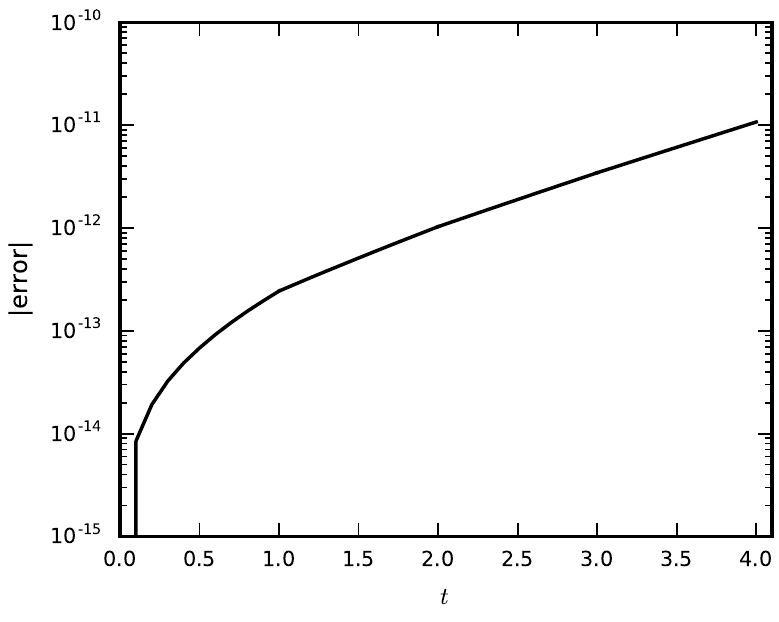}
    \end{minipage}

    \begin{center}
    \caption{\label{fig:bad}
        The errors for using the classic $4$th order explicit Runge-Kutta method (Left) and explicit SDG method with $p=3$ (right) for equation~\eqref{eq:bad}. The upper figures correspond to $\Delta t = 0.2$ and the lower to $\Delta t = 0.1$.
}
     \end{center}
\end{figure}

\subsection{Linear Convection Equation}
The third example is the application with method of lines approach for PDEs.  Here, we consider a linear scalar convection equation 
\begin{equation}
  \left\{\begin{array}{lc}
    u_t + u_x = 0, & (x,t) \in [0,1]\times[0,T] \\
    u(x,0) = \sin(2\pi x), & 
  \end{array}\right.
  \label{eq:convection}
\end{equation}
with periodic boundary conditions. The errors are computed at $T = 1$ which is one period in time.

For spatial discretization, we consider the technique used in~\cite{Cockburn:2003} that first uses the DG method to get the approximation, then a postprocessing technique is applied to enhance spatial accuracy.  Here, we use the polynomial basis of degree $4$ and $160$ elements for the spatial discretization.  In this case, the spatial accuracy in the $L^2$ norm has order $5$ for the DG approximation and order $9$ after postprocessing.  For the detail of this postprocessing technique or say spatial superconvergence, we refer to~\cite{Bramble:1977, Cockburn:2003, Li:2015, Ryan:2015, Li:2016}.  For time discretization, since it is a convection problem with spatial accuracy order $9$, we simply use the $ExSDG_4^8$ method with the time step $\Delta t = \text{cfl}\cdot\Delta x$ where the CFL number is $0.1$.  We compare the SDG method with the third order Runge-Kutta method as in~\cite{Cockburn:2003}. We present the errors in the $L^2$-norm with different CFL numbers in Figure~\ref{fig:convection}.  In Figure~\ref{fig:convection}, we can see that the explicit SDG method has better stability than the Runge-Kutta method as the Runge-Kutta method is not stable for CFL number $0.1$.  More important, the Runge-Kutta method requires a much smaller CFL number in order to achieve the desired accuracy.  In Figure~\ref{fig:convection}, to achieve the same accuracy after postprocessing, the Runge-Kutta method has to use more than $5,000$ times smaller time step size compared to the explicit SDG method. In terms of computational time, using the Runge-Kutta method is more than $100$ times slower.  In addition, we provide Figure~\ref{fig:ordercfl} to demonstrate the relation between the achieved accuracy order and the CFL number.  We note that the further studies and applications of the SDG method with this postprocessing technique will be studied in a forthcoming paper.

\begin{figure}[!ht]
    \centering
    \begin{minipage}[c]{0.48\textwidth}
      \centerline{$\qquad\quad \text{Before Postprocessing}$}
    \end{minipage}
    \begin{minipage}[c]{0.48\textwidth}
      \centerline{$\qquad\quad \text{After Postprocessing}$}
    \end{minipage}
    \vspace{-0.5cm}
    \begin{minipage}[c]{0.48\textwidth}
    \includegraphics[width=1.\textwidth]{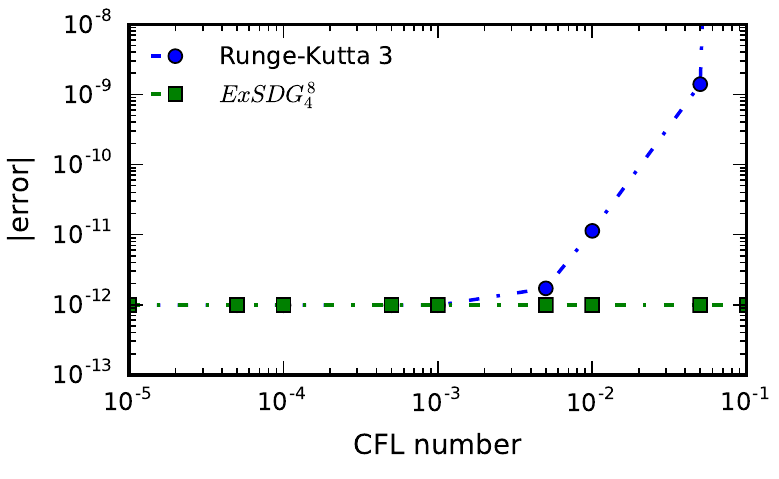}
    \end{minipage}
    \begin{minipage}[c]{0.48\textwidth}
    \includegraphics[width=1.\textwidth]{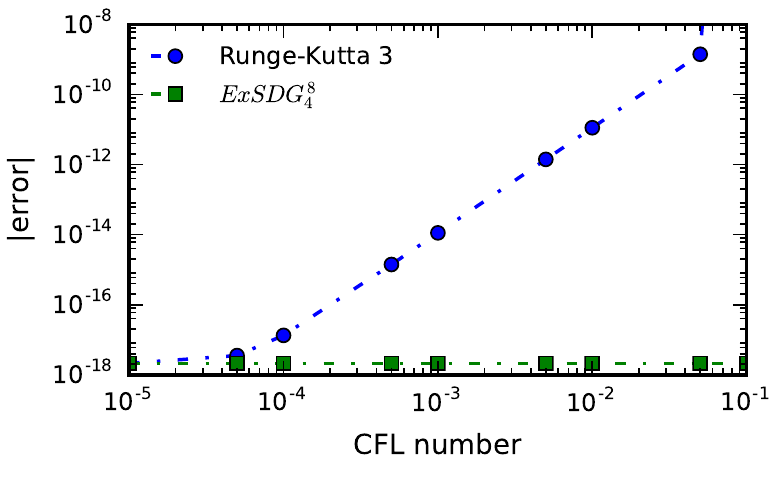}
    \end{minipage}
    
    \begin{center}
    \caption{\label{fig:convection}
        The error using the third order Runge-Kutta method and $ExSDG_4^8$ method with different CFL number for the linear convection equation~\eqref{eq:convection}. 
}
     \end{center}
\end{figure}

\begin{figure}[!ht]
    \centering
    \begin{minipage}[c]{0.48\textwidth}
      \centerline{$\qquad\quad \text{Before Postprocessing}$}
    \end{minipage}
    \begin{minipage}[c]{0.48\textwidth}
      \centerline{$\qquad\quad \text{After Postprocessing}$}
    \end{minipage}
    \vspace{-0.5cm}
    \begin{minipage}[c]{0.48\textwidth}
    \includegraphics[width=1.\textwidth]{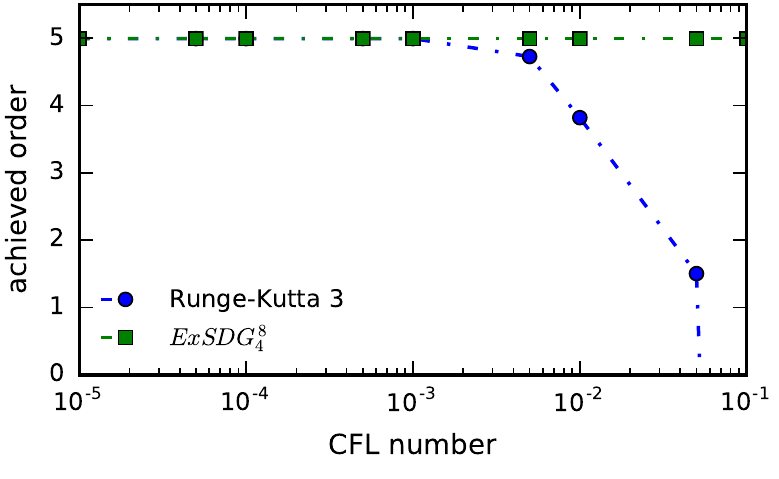}
    \end{minipage}
    \begin{minipage}[c]{0.48\textwidth}
    \includegraphics[width=1.\textwidth]{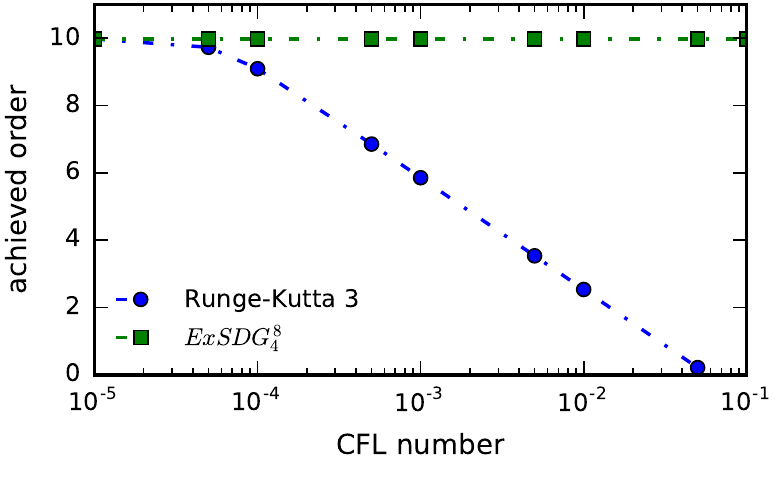}
    \end{minipage}
    
    \begin{center}
    \caption{\label{fig:ordercfl}
        The achieved accuracy order using the third order Runge-Kutta method and $ExSDG_4^8$ method with different CFL number for the linear convection equation~\eqref{eq:convection}.  Here, the accuracy order is computed based on the reference solutions solved on a uniform mesh with $20$ elements. 
}
     \end{center}
\end{figure}

\section{Conclusion}
In this paper, we introduce a new class of iterative methods (SDG) for ODEs based on the DG methods. Compared to the existed time integrators, we demonstrate that the proposed SDG schemes have several advantages:
\begin{itemize}
    \item The explicit, implicit and semi-implicit schemes can be systematically constructed for arbitrary order of accuracy. 
    \item Theoretically analysis shows that the SDG schemes have the accuracy order of $\min\{2p+1, K+1\}$ which preserves the superconvergence property of the DG methods.  
    \item For stability, the implicit SDG schemes are numerical explored to be $A$-stable for very high order schemes, while the explicit schemes have reasonable stability for non-stiff or mid-stiff problems.  In addition, the $L$-stable scheme can be easily constructed based the implicit scheme. 
    \item The SDG schemes can be easily combined with the method of lines approach to generate a framework for space-time discretizations.
    \item The SDG schemes can be naturally integrated with multilevel technique.  The multilevel SDG methods can be easily coupled with existed spatial multigrid methods to create space-time multigrid framework. 
    \item The schemes inherit the well studied $hp$-adaptive strategies from DG methods and other one-step time integrators.  Moreover, a space-time adaptive strategy can be constructed. 
    \item The discontinuous nature of the schemes give a more flexible structure to deal the difficulties raise from ``bad'' problems, such as discontinuities.
\end{itemize}
  
Besides the direct applications for ODEs and PDEs,  the SDG methods are also competitive and promising candidates for time parallel algorithms such as Parareal, PFASST, etc.  The application in a time parallel setting will be discussed in our upcoming work.

\section*{Acknowledgements}
We would like to thank Dr. Martin Weiser for fruitful discussions.
This work was supported by the Swiss Platform for Advanced Scientific Computing (PASC), under the project ``Integrative HPC Framework for Coupled Cardiac Simulations'', the Swiss National Science Foundation (SNF) and the Deutsche Forschungsgemeinschaft in the framework of the project “ExaSolvers - Extreme Scale Solvers for Coupled Systems, SNF project numbers 145271 and 162199, within the DFG-Priority Research Program 1684 ``SPPEXA- Software for Exascale Computing'', and the SCCER SoE (Swiss Competence Center for Energy Research - Supply of Energy).

\bibliography{ref}
\bibliographystyle{plain}

\end{document}